\documentclass[a4paper,12pt,reqno]{amsart}

\usepackage{amsmath,amssymb,amsfonts}
\usepackage{graphicx,color}
\usepackage[colorlinks=true,linkcolor=blue,citecolor=red]{hyperref}

\numberwithin{equation}{section}

\newtheorem{thm}{Theorem}[section]
\newtheorem{prop}[thm]{Proposition}
\newtheorem{lem}[thm]{Lemma}

\theoremstyle{definition}
\newtheorem{defn}[thm]{Definition}

\theoremstyle{remark}
\newtheorem{rem}[thm]{Remark}

\allowdisplaybreaks[4]

\newcommand{\al}{\alpha}
\newcommand{\be}{\beta}
\newcommand{\ga}{\gamma}

\newcommand{\la}{\lambda}

\newcommand{\om}{\omega}

\renewcommand{\th}{\theta}
\newcommand{\ze}{\zeta}

\newcommand{\p}{\partial}

\newcommand{\Sc}[1]{\mathcal{#1}}

\newcommand{\Bo}[1]{\mathbb{#1}}
\newcommand{\R}{\Bo{R}}

\newcommand{\HP}{\Bo{P}}

\newcommand{\bbar}{\overline}
\newcommand{\ti}{\widetilde}

\newcommand{\eq}[2]{\begin{equation} \label{#1} \begin{split} #2 \end{split} \end{equation}}

\newcommand{\eqq}[1]{\begin{align*} #1 \end{align*}}
\newcommand{\eqs}[1]{\begin{gather*} #1 \end{gather*}}
\newcommand{\mat}[1]{\begin{smallmatrix} #1 \end{smallmatrix}}

\newcommand{\hx}{\hspace{10pt}}

\newcommand{\gaiseki}{\!\times\!}

\renewcommand{\Re}{\mathrm{Re}\,}
\renewcommand{\Im}{\mathrm{Im}\,}

\newcommand{\mbx}{\mathbf{x}}
\newcommand{\mbn}{\mathbf{n}}
\newcommand{\mbm}{\mathbf{m}}
\newcommand{\mbe}{\mathbf{e}}
\newcommand{\mbu}{\mathbf{u}}

\newcommand{\mbz}{\mathbf{0}}
\newcommand{\mbom}{{\boldsymbol \om}}

\begin{document}

\title[3D Euler flows with finite Fourier modes]
{Characterization of three-dimensional Euler flows supported on finitely many Fourier modes} 

\author{Nobu Kishimoto}
\address{Research Institute for Mathematical Sciences, Kyoto University, Kyoto 606-8502, Japan} 
\email{nobu@kurims.kyoto-u.ac.jp} 

\author{Tsuyoshi Yoneda}
\address{
Graduate School of Economics, Hitotsubashi University,
Tokyo 186-8601, Japan 
} 
\email{t.yoneda@r.hit-u.ac.jp}

\date{\today} 


\maketitle

\newcommand{\Stwo}{\mathbb{S}^2}
\newcommand{\cR}[1]{\Sc{R}_{#1}}


\begin{abstract} 
Recently, the Nash-style convex integration has been becoming 
the main scheme for the mathematical study of turbulence,
and the main building block of it has been either Beltrami flow (finite mode) or Mikado flow (compactly supported in the physical side). 
On the other hand, in physics, it is observed that turbulence is composed of a hierarchy of scale-by-scale vortex stretching.
Thus our mathematical motivation in this study is to find another type of building blocks accompanied by
vortex stretching and scale locality (possibly finitely many Fourier modes).
In this paper, we give a complete list of solutions to the 3D Euler equations with finitely many Fourier modes, which is an extension of the corresponding 2D result by Elgindi-Hu-\v{S}ver\'ak (2017).
In particular, we show that there is no 3D Euler flows with finitely many Fourier modes, except for stationary 2D-like flows and Beltrami flows. 
We also discuss the case when viscosity and Coriolis effect are present.
\end{abstract}

\section{Introduction}

Recent DNS
\cite{Goto-2008,GSK,Motoori-2019,Motoori-2021} of
turbulence at sufficiently high Reynolds numbers have reported that
there exists a hierarchy of vortex stretching motions in developed turbulence.
In particular, Goto-Saito-Kawahara \cite{GSK} clearly observed that turbulence at sufficiently high Reynolds numbers in a periodic cube is composed of a self-similar
hierarchy of antiparallel pairs of vortex tubes, and  it is sustained
by creation of smaller-scale vortices due to stretching in larger-scale strain fields. 
They also observed that vortices at each hierarchical level are most likely to
be stretched in strain fields around two to eight times larger vortices (we  call it scale locality). This observation is further investigated by Y-Goto-Tsuruhashi \cite{YGT} (see also \cite{TGOY}). Thus we could conclude physically  
that  local-scale energy transfer is mainly induced by vortex stretching, and in mathematics, the following question naturally arises (see also \cite{JY1,JY2,JY3} for the related results):

\vspace{0.4cm}

\begin{quotation}
\begin{center}
``Can we construct a solution (locally in scale) to the  incompressible Euler equations accompanied by vortex stretching, as a concrete picture of the hierarchy of turbulence?"
\end{center}
\end{quotation}

\vspace{0.4cm}

Nowadays, the Nash-style convex integration has been becoming the main scheme for 
the mathematical study of turbulence.
This scheme was first initiated by De Lellis and Sz{\'e}kelydihi Jr. \cite{DS}.
They showed the existence of a $C^{0+}_{x,t}$ weak solution of the 3D Euler equations which is non-conservative, following the Nash scheme with Beltrami building blocks.
After several results appears,  Isett \cite{I} showed existence of dissipative weak solutions in the regularity class $C^{1/3-}_{x,t}$ by using the Mikado flows, as building blocks. Thus  in the Nash scheme, Beltrami flows and Mikado flows (both are stationary Euler flows) are the elementary pieces  in multi-scale ideal turbulence.
More precisely, in the Nash scheme,  we need to construct a sequence of triplets $(v_q,p_q,\mathring R_q)_{q=1}^\infty$ solving the following Euler-Reynolds system (see\cite{BLIS}):
\begin{equation*}
\begin{split}
\partial_tv_q+\nabla\cdot (v_q\otimes v_q)+\nabla p_q&=\nabla\cdot  \mathring{R}_q\\
\nabla\cdot v_q&=0.
\end{split}
\end{equation*}
And then we set the perturbation $w_q:=v_q-v_{q-1}$ as the following:
\begin{equation*}
w_q(x,t)=\sum_ka_{k,q}(x,t)\phi_{k,q}(x,t)B_{k,q}e^{i\lambda_qk\cdot x},
\end{equation*}
where $a_{k,q}$ is the amplitude, $\phi_{k,q}$ is a phase function and $B_{k,q}e^{i\lambda_q k\cdot x}$
is a complex Beltrami mode at frequency $\lambda_q$ (in this case, this is the building block).
Thus in this Nash scheme, we need to choose appropriate $a_{k,q}$ and $\phi_{k,q}$ in order to have  $\mathring R_q\to 0$ in a weak sense.

Clearly, the antiparallel pair of vortex tubes with stretching motion (locally in scale) is neither Beltrami flow nor Mikado flow,
thus, in order to construct a concrete picture of turbulence by using this Nash scheme, we need to solve the following question: 

\vspace{0.4cm}

\begin{quotation}
\begin{center}
``Can we find another type of Euler flows as building blocks (possibly finitely many Fourier modes), to construct the dissipative Euler solutions?"
\end{center}
\end{quotation}

\vspace{0.4cm}

In this paper, we give a partial answer to it, namely, we show that, at least,  finite-mode (stationary or non-stationary) Euler flow does not exist
except for stationary 2D-like flows and Beltrami flows. 
This means that we cannot construct any Euler flow accompanied by vortex stretching supported on finitely many
Fourier modes. 
In what follows, let us formulate this partial answer more precisely.

The incompressible Euler equations on $\R^3$ are expressed as follows:
\eq{Euler}{\p _t\mathbf{u}+(\mathbf{u}\cdot \nabla )\mathbf{u}+\nabla p=\mbz ,\qquad \nabla \cdot \mathbf{u}=0,}
and we characterize (real-valued) solutions of the form
\eqs{
\mathbf{u}(t,\mbx)=\sum _{\mbn \in S}\mathbf{u}_\mbn(t)e^{i\mbn \cdot \mbx},\qquad p(t,\mbx )=\sum _{\mbn \in S_p}p_\mbn (t)e^{i\mbn \cdot \mbx};\\
\mathbf{u}_\mbn:I\to \mathbb{C}^3,\quad p_\mbn:I\to \mathbb{C},\quad \text{$I\subset \R$: open interval},\quad \text{$S$, $S_p$: \emph{finite} subsets of $\mathbb{R}^3$}.
}
Note that the corresponding problem in 2D was already answered by Elgindi-Hu-\v{S}ver\'ak~\cite{EHS17} (see Theorem~\ref{thm:2D} below).
We do not restrict the frequency to a lattice, so the solutions we consider are in general spatially quasi-periodic.

The real-valuedness implies that $S$ should be symmetric (i.e., $-S=S$) and $\mathbf{u}_{-\mbn}=\overline{\mathbf{u}_\mbn}$ for $\mbn\in S$, and similarly for $S_p$ and $\{ p_\mbn\} _{\mbn\in S_p}$.
Since the zero mode of $p$ is not relevant, we may assume $\mbz \not\in S_p$.
If $\mbz \in S$, we see from the equation that the zero mode (or the spatial mean) is independent of $t$; $\mathbf{u}_\mbz (t)\equiv \mathbf{u}_\mbz$.
Then, $(\mathbf{v},q)$ defined by 
\eqs{\mathbf{v}(t,\mbx)=\mathbf{u}\big( t, \mbx+t\mathbf{u}_\mbz \big) -\mathbf{u}_\mbz =\sum _{\mbn\in S\setminus \{ \mbz \}} \mathbf{u}_\mbn(t)e^{it\mbn\cdot \mathbf{u}_\mbz}e^{i\mbn\cdot \mbx},\\
q(t,\mbx)=p\big( t, \mbx+t\mathbf{u}_\mbz \big) =\sum _{\mbn\in S_p}p_\mbn(t)e^{it\mbn\cdot \mathbf{u}_\mbz}e^{i\mbn\cdot \mbx}
}
is a mean-zero solution of \eqref{Euler} with the same Fourier support.
Conversely, for any given mean-zero solution $(\mathbf{v},q)$ and any $\mathbf{u}_\mbz \in \mathbb{R}^3$, we obtain a solution $(\mathbf{u},p)$ with $\mathbf{u}_\mbz (t)\equiv \mathbf{u}_\mbz$ by inverting the above transformation.
Therefore, it suffices to characterize mean-zero solutions supported on finitely many Fourier modes.

\begin{defn}
Let $\Sc{H}$ denote the set of all real-valued divergence-free vector fields with finitely many Fourier modes; that is,
\eqq{\Sc{H}:=\big\{ \mathbf{u}(x)=\sum _{\mbn\in S}\mathbf{u}_\mbn e^{i\mbn\cdot \mbx} ~:~&S\subset \R ^3\setminus \{ \mbz \} ~\text{finite and symmetric},\\[-10pt]
&\text{$\mathbf{u}_\mbn\in \Bo{C}^3\setminus \{ \mbz \}$,\hx $\mathbf{u}_{-\mbn}=\bbar{\mathbf{u}_\mbn}$,\hx $\mbn \cdot \mathbf{u}_\mbn=0$\hx ($\forall \mbn\in S$)}\big\} .}
Here and in the sequel, we also denote by ``$\cdot$'' the dot product for vectors in $\mathbb{C}^3$ (it is $\Bo{C}$-bilinear and different from the inner product of $\Bo{C}^3$ which is sesquilinear).
We also define the corresponding set $\Sc{H}_I$ of space-time functions on $I\times \mathbb{R}^3$, for an open interval $I\subset \R$, by
\eqq{\Sc{H}_I:=\{ \mathbf{u}(t,\mbx)=\sum _{\mbn \in S}\mathbf{u}_\mbn (t)e^{i\mbn \cdot \mbx} ~:~&S\subset \R ^3\setminus \{ \mbz \} ~\text{finite and symmetric},\hx \mathbf{u}_\mbn :I\to \Bo{C}^3,\\[-10pt]
&\text{$\mathbf{u}_\mbn \not\equiv \mbz$,\hx $\mathbf{u}_{-\mbn}=\bbar{\mathbf{u}_\mbn}$,\hx $\mbn \cdot \mathbf{u}_\mbn =0$\hx ($\forall t\in I$, $\forall \mbn \in S$)}\} .}
The set $S$ is called the Fourier support of $\mathbf{u}$.
\end{defn}
As in the definition of $\Sc{H}$ and $\Sc{H}_I$, 
by writing $\mathbf{u}(x)=\sum _{\mbn \in S}\mathbf{u}_\mbn e^{i\mbn \cdot \mbx}$ we normally assume that the coefficient vectors are nonzero for $\mbn \in S$, and we use the convention that $\mathbf{u}_\mbn=\mbz $ if $\mbn \not\in S$.
For $\mathbf{u}(t,\mbx )=\sum _{\mbn \in S}\mathbf{u}_\mbn (t)e^{i\mbn \cdot \mbx}\in \Sc{H}_I$, we see that
\eqq{(\mathbf{u}\cdot \nabla )\mathbf{u}(\mbx )=\frac{i}{2}\sum _{\mat{\mbn \in S+S\\ \mbn \neq \mbz}}e^{i\mbn \cdot \mbx}\sum _{\mat{\mbn_1,\mbn_2\in S\\ \mbn_1+\mbn_2=\mbn}}\big[ (\mathbf{u}_{\mbn_1}\cdot \mbn_2)\mathbf{u}_{\mbn_2}+(\mathbf{u}_{\mbn_2}\cdot \mbn_1)\mathbf{u}_{\mbn_1}\big] .
}
(Note that the zero mode does not appear due to the divergence-free condition.)
Hence, if $(\mathbf{u},p)$ is a solution to \eqref{Euler} on $I\times \R ^3$, then the pressure $p$ has the expression $p=\sum _{\mbn\in S_p}p_\mbn e^{i\mbn\cdot \mbx}$ with $S_p\subset \ti{S}:=S\cup \big[ (S+S)\setminus \{ \mbz \}\big]$ and
\eqq{\p _t\mathbf{u}_\mbn+\frac{i}{2}\sum _{\mat{\mbn_1,\mbn_2\in S\\ \mbn_1+\mbn_2=\mbn}}\big[ (\mathbf{u}_{\mbn_1}\cdot \mbn_2)\mathbf{u}_{\mbn_2}+(\mathbf{u}_{\mbn_2}\cdot \mbn_1)\mathbf{u}_{\mbn_1}\big] =-ip_\mbn \mbn,\qquad t\in I,\quad \mbn\in \ti{S}.}
For $\mbn \in \R^3\setminus \{ \mbz \}$, let $\hat{\HP}_{\mbn}$ denote the orthogonal projection in $\Bo{C}^3$ onto the two-dimensional subspace $\{ \mathbf{v}\in \mathbb{C}^3:\mbn \cdot \mathbf{v}=0\}$, so that $\hat{\HP}_\mbn$ is the representation of the Helmholtz projection on the Fourier side.
Since $\hat{\HP}_\mbn \mathbf{u}_\mbn=\mathbf{u}_\mbn$ and $\hat{\HP}_\mbn \mbn=\mbz$, it holds that
\begin{alignat}{2}
&\p _t\mathbf{u}_\mbn + \frac{i}{2}\hat{\HP}_\mbn \sum _{\mat{\mbn_1,\mbn_2\in S\\ \mbn_1+\mbn_2=\mbn}}\big[ (\mathbf{u}_{\mbn_1}\cdot \mbn_2)\mathbf{u}_{\mbn_2}+(\mathbf{u}_{\mbn_2}\cdot \mbn_1)\mathbf{u}_{\mbn_1}\big] =\mbz ,&\qquad &t\in I,\quad \mbn\in \ti{S}, \label{cond:Euler}\\
&\frac{i}{2}\big( \mathrm{Id}-\hat{\HP}_\mbn \big) \sum _{\mat{\mbn_1,\mbn_2\in S\\ \mbn_1+\mbn_2=\mbn}}\big[ (\mathbf{u}_{\mbn_1}\cdot \mbn_2)\mathbf{u}_{\mbn_2}+(\mathbf{u}_{\mbn_2}\cdot \mbn_1)\mathbf{u}_{\mbn_1}\big] =-ip_\mbn \mbn, & &t\in I,\quad \mbn\in \ti{S}. \label{pressure:Euler}
\end{alignat}
Once we obtain a solution $\{ \mathbf{u}_\mbn (t)\}_{\mbn \in S}$ of \eqref{cond:Euler}, $\{ p_\mbn (t)\}_{\mbn\in S_p}$ is determined by \eqref{pressure:Euler}; in fact, it is given by
\[ p_\mbn (t)=-\frac{1}{|\mbn|^2}\sum _{\mat{\mbn_1,\mbn_2\in S\\ \mbn_1+\mbn_2=\mbn}}(\mathbf{u}_{\mbn_1}(t)\cdot \mbn_2)(\mathbf{u}_{\mbn_2}(t)\cdot \mbn_1),\qquad t\in I,\quad \mbn\in \ti{S}.\]
Hence, it suffices to give a characterization of $\mathbf{u}\in \mathcal{H}_I$ for which $\{ \mathbf{u}_\mbn (t)\}_{\mbn \in S}$ solves the ODE system \eqref{cond:Euler}.
From now on, we consider the equation \eqref{cond:Euler} instead of \eqref{Euler}.
As observed in the 2D case \cite{EHS17}, a vector field $\mathbf{u}\in \Sc{H}_I$ solving \eqref{cond:Euler} has the Fourier coefficient vectors $\{ \mathbf{u}_\mbn (t)\}$ each of which has components real analytic in time.
In particular, the set
\eq{def:I_0}{I_0:=\{ t\in I: \mathbf{u}_\mbn (t)=\mbz ~~\text{for some $\mbn\in S$}\}}
cannot have an accumulation point in $I$.
Note that the Fourier support of $\mathbf{u}(t,\cdot )$ coincides with the set $S$ if $t\in I\setminus I_0$.

Before stating our result, let us recall some basic facts on Beltrami flows.
We call a (divergence-free) eigenfield of the curl operator a Beltrami flow; that is, $\mathbf{b}\in \Sc{H}$ is a Beltrami flow if
\eqq{\exists \lambda \in \Bo{C};\quad \nabla \times \mathbf{b}=\lambda \mathbf{b},
}
where ``$\times$'' stands for the cross product.
A Beltrami flow $\mathbf{b}(\mbx )$ is always a stationary solution to \eqref{Euler} with $p=-|\mathbf{b}|^2/2$, and also an eigenfield of the Laplace operator $\Delta$ with eigenvalue $-\lambda ^2$.
The following lemma gives a characterization of the Beltrami flows (for the proof, see, e.g., the argument in \cite[Section III]{MP91}):
\begin{lem}\label{lem:Beltrami}
A vector field $\mathbf{b}\in \Sc{H}$ is a Beltrami flow (with eigenvalue $\lambda$) if and only if $\lambda \in \R \setminus \{ 0\}$ and the following conditions (i)--(ii) hold:
\begin{enumerate}
\item The Fourier support of $\mathbf{b}$ is a subset of the sphere with radius $|\la |$ centered at the origin.
\item For each frequency $\mbn$ in the Fourier support of $\mathbf{b}$, the coefficient vector $\mathbf{b}_\mbn$ at $\mbn$ satisfies
$|\Re \mathbf{b}_\mbn|=|\Im \mathbf{b}_\mbn|$ and $\Re \mathbf{b}_\mbn \cdot \Im \mathbf{b}_\mbn =0$.
Moreover, $(\mbn,\Re \mathbf{b}_\mbn,\Im \mathbf{b}_\mbn)$ forms a right-handed (resp. left-handed) system if $\la >0$ (resp. $\la <0$).
\end{enumerate}
\end{lem}
Note that the condition (i) characterizes eigenfields of the Laplacian, and the additional condition (ii) on the coefficient vectors is needed for a characterization of eigenfields of the curl. 
The so-called Arnold-Beltrami-Childress flow (ABC flow)
\[ \mbu (\mbx )=\Big( B\cos x^2+C\sin x^3,\,C\cos x^3+A\sin x^1,\,A\cos x^1+B\sin x^2\Big) ,\quad A,B,C\in \mathbb{R}, \]
which is also known as a stationary Euler flow, is an example of Beltrami flows.
In fact, its Fourier support is $\{ \pm (1,0,0),\pm(0,1,0),\pm (0,0,1)\}$ and the coefficient vectors are given by
\[ \mbu_{\pm (1,0,0)}=A\Big( 0, \frac{\pm 1}{2i}, \frac{1}{2}\Big) ,\quad \mbu_{\pm (0,1,0)}=B\Big( \frac{1}{2}, 0, \frac{\pm 1}{2i}\Big) ,\quad \mbu_{\pm (0,0,1)}=C\Big( \frac{\pm 1}{2i}, \frac{1}{2}, 0\Big) , \]
from which we easily see that the above conditions (i), (ii) are satisfied (with $\lambda =+1$).

The specific aim of this article is to give the complete list of the vector fields in $\Sc{H}_I$ solving \eqref{cond:Euler}.
Let us recall a characterization of such flows in 2D given in \cite{EHS17}:
\begin{thm}[{\cite[Theorem~5.1]{EHS17}}]\label{thm:2D}
If $\mathbf{u}(t,\mbx)$ is a (real-valued, mean-zero) solution of 2D incompressible Euler equations which is supported on finitely many Fourier modes, then $\mathbf{u}$ is independent of time.
Moreover, its Fourier support is either a subset of a circle centered at the origin, or a line passing through the origin.
\end{thm}
It is easy to see that the above necessary condition on the shape of the Fourier support is also sufficient for a vector field to be a stationary Euler flow.
In particular, there is no restriction on the coefficient vector of each Fourier mode.

The 2D result, Theorem~\ref{thm:2D}, was shown by the following steps:

\begin{itemize}
\item[Step 0:]
Two modes $\mathbf{u}_\mbn e^{i\mbn\cdot \mbx}$, $\mathbf{u}_\mbm e^{i\mbm \cdot \mbx}$ do not contribute to the mode $\mbn+\mbm$ through nonlinear interaction if and only if either $\mbn$ is parallel to $\mbm$ or $|\mbn|=|\mbm|$.

This can be easily shown by the vorticity representation.
(Notice here that the coefficient vectors $\mathbf{u}_\mbn,\mathbf{u}_\mbm$ are not relevant.)
\item[Step 1:]
Let $S^{conv}$ be the convex hull of the Fourier support $S$ of a solution, and assume that $S^{conv}$ is not contained in a line.
Then, all the vertices of $S^{conv}$ are located on a circle centered at the origin and there is no point on the boundary of $S^{conv}$ other than its vertices.

This can be verified from Step~0 and the following fact:
There is no contribution from the nonlinear interaction between any vertex of $S^{conv}$ and its adjacent point in $S$ on the boundary of $S^{conv}$ (which may be an adjacent vertex or a point on a side of $S^{conv}$).
This is roughly because the mode created by such two points could not be in $S$ and could not be achieved by any other pairs of two points in $S$.
\item[Step 2:]
Any two vertices of $S^{conv}$ (not necessarily adjacent) create no contribution through nonlinear interaction.

In 2D this is an immediate consequence of Step~0 and Step~1.
\item[Step 3:]
There is no point of $S$ in the interior of $S^{conv}$.
In particular, Step~2 shows that any two points in $S$ do not interact, and then the solution must be stationary.

This is shown via contradiction argument:
If not, then the interior point $\mbm \in S$ which is the ``farthest from the origin'' would interact with the point $\mbn \in S$ which is one of its ``nearest vertices'' of $S^{conv}$ to create nonzero contribution at a point outside $S$. Therefore there would be other pair(s) of points in $S$ which cancel it out.
However, all the other possible pairs would have to consist of two vertices of $S^{conv}$ by the definition of $\mbm$, and hence, by Step~2, create no output. 
We would thus come to a contradiction.
\end{itemize}

In 3D, the Euler flows with finitely many Fourier modes include at least the following stationary examples:
\begin{itemize}
\item $\mathbf{u}(\mbx)=(0,u^2(x^1),u^3(x^1))$, $p=0$, where $u^2,u^3$ are any functions supported on finitely many Fourier modes.
\item $\mathbf{u}(\mbx)=(0,0,u^3(x^1,x^2))$, $p=0$, where $u^3$ is any function supported on finitely many Fourier modes.
\item 2D flows: $\mathbf{u}(\mbx)=(u^1(x^1,x^2),u^2(x^1,x^2),0)$, $p(\mbx)=p(x^1,x^2)$, where $(u_1,u_2,p)$ is a 2D Euler flow supported on finitely many Fourier modes.
(A characterization of such a flow is given in Theorem~\ref{thm:2D}.)
\item Beltrami flows.
\end{itemize}
In a sharp contrast to the 2D case, any condition on the shape of the Fourier support is not enough by itself for characterization of the whole solutions supported on finitely many Fourier modes, and an extra condition on the coefficient vectors should be required.
Indeed, the Fourier support of a Beltrami flow is on a sphere centered at the origin (similarly to the circle condition in 2D), but not all such vector fields (i.e., divergence-free eigenfields of the Laplacian) are stationary Euler flows.

Now, we state our main result:
\begin{thm}\label{thm:main}
$\mathbf{u}=\mathbf{u}(t,\mbx)\in \Sc{H}_I$ is a solution of \eqref{cond:Euler} on an open interval $I\subset \R$ if and only if it is independent of time and satisfies one of the following:
\begin{enumerate}
\item $S$, the Fourier support of $\mathbf{u}(\mbx)\in \Sc{H}$, is a subset of a line passing through the origin.
\item $S$ has two linearly independent points and is a subset of a plane $P$ containing the origin.
Moreover, one of the following holds:
\begin{enumerate}
\item $\mathbf{u}(\mbx)$ is perpendicular to $P$ everywhere.
\item $\mathbf{u}(\mbx)=\mathbf{u}^\parallel (\mbx )+u^\perp (\mbx)\mbe^\perp$, where $\mbe^\perp$ denotes (one of) the unit normal vector to $P$ and $\mathbf{u}^\parallel:\R ^3 \to \R ^3$, $u^\perp:\R ^3\to \R$ satisfy the following:
\begin{itemize}
\item $\mathbf{u}^\parallel (\mbx)\in \Sc{H}$ is parallel to $P$ everywhere, and its Fourier support is a subset of a circle on $P$ (with radius $\lambda >0$) centered at the origin and contains at least four points.
\item $u^\perp(\mbx)=Q(\om (\mbx))-\big\langle Q(\om )\big\rangle$, where $\om (\mbx)$ is the scalar function defined by $\nabla \times \mathbf{u}^\parallel =\lambda \om \mbe^\perp$, $Q$ is a polynomial with real coefficients, and $\big\langle Q(\om )\big\rangle$ denotes the zero mode (mean value) of $Q(\om (\mbx ))$.
\end{itemize}
\end{enumerate}
\item $S$ has three linearly independent points and is a subset of a sphere centered at the origin, and $\mathbf{u}(\mbx)$ is a Beltrami flow.
\end{enumerate}
\end{thm}

We make some remarks on Theorem~\ref{thm:main}.

\begin{rem}
The flows of type (i) and (ii)-(a) are obtained by rotating those of the form $\mathbf{u}(\mbx)=(0,u^2(x^1),u^3(x^1))$ and $(0,0,u^3(x^1,x^2))$, respectively, while rotation of the 2D Euler flow $\mathbf{u}(\mbx)=(u^1(x^1,x^2),u^2(x^1,x^2),0)$ gives the flow of type (ii)-(b) with $Q=0$.
Hence, except for the obvious examples mentioned above, the flow of type (ii)-(b) with $Q\neq 0$ is the only possibility for the Euler flow with finite Fourier support.
In particular, it turns out that Beltrami flows are the only genuinely 3D Euler flows with finitely many Fourier modes.

The flow of type (ii)-(b) with $Q\neq 0$ is obtained as a rotation of a two-dimensional and three-component flow $\mbu (x^1,x^2)$; the theorem says that the vertical component $u^3(x^1,x^2)$ of such a solution (with finite Fourier support) must be a polynomial of the vorticity field associated with the horizontal component $(u^1(x^1,x^2),u^2(x^1,x^2))$.
We also notice that a flow of type (ii)-(b) is a Beltrami flow if and only if $Q(\om )=\pm \om$.
\end{rem}

\begin{rem}
The 2D result (Theorem~\ref{thm:2D}) remains true for complex-valued solutions, as mentioned in \cite{EHS17}.
This is, however, not the case in 3D; at some points of the proof of Theorem~\ref{thm:main} we will take advantage of real-valuedness in a more crucial way.
In fact, the complex-valued vector field 
\begin{align*}
\mbu (\mbx )&= \frac{1}{2i}\sum _{\sigma \in \{ \pm 1\}}\Big\{ \sigma (1,-1,0)e^{i(\sigma ,\sigma ,1)\cdot \mbx} +\sigma (1,1,0)e^{i(\sigma ,-\sigma ,-1)\cdot \mbx}\Big\} \\
&=\Big( e^{ix^3}\sin (x^1+x^2)+e^{-ix^3}\sin (x^1-x^2),\,-e^{ix^3}\sin (x^1+x^2)+e^{-ix^3}\sin (x^1-x^2),\,0\Big) 
\end{align*}
is a stationary solution to \eqref{Euler} with $p(\mbx)=\cos (2x^1)+\cos (2x^2)$, and its Fourier support consists of four points forming a regular tetrahedron, but it is not a Beltrami flow.
\end{rem}

\begin{rem}
In the 2D case \cite{EHS17}, the characterization given in Theorem~\ref{thm:2D} was used as  one of the key tools to investigate long-time behavior of solutions to \textit{partially undamped} Navier-Stokes equations on the 2D torus. 
With our theorem, part of their result may be extended to the 3D case, provided that the equation is appropriately modified so that it admits global strong solutions.
We plan to address this problem in a forthcoming paper.
\end{rem}

%

Here are some comments on the proof of Theorem~\ref{thm:main}.
In 2D the divergence-free condition reduces the problem to the scalar equation; while in 3D, two-dimensional degree of freedom still remains for each Fourier coefficient, which makes the argument substantially more involved. 
A basic tool is the characterization of two modes not interacting with each other, which corresponds to Step~0 in the 2D case and will be discussed in Section~\ref{sec:two}.
A straightforward calculation gives an algebraic description (Lemma~\ref{lem:two} below), and we also give a geometric interpretation using a rotation operator (Proposition~\ref{prop:two} below).
The algebraic characterization will be exploited to treat the planar case (i.e., the situation (ii) in the theorem) in Section~\ref{sec:2D}; the analysis on the horizontal component $\mbu^\parallel$ is similar to the proof for the 2D result in \cite{EHS17} based on elementary convex analysis, while the narrowing-down argument for the vertical component $u^\perp$ seems new and of its own interest.
The main novelty is the reduction to Beltrami flows in the 3D case (i.e., the situation (iii) in the theorem) to be presented in Section~\ref{sec:S}, where the geometric characterization of two non-interacting modes will be combined with the Gauss-Bonnet theorem to play a vital role.

Finally, we point out that our approach remains valid under the presence of viscosity and Coriolis effect.
Although it is not the main goal of this paper, as an interesting generalization of Theorem~\ref{thm:main} we will include the precise statement of characterization of finite-mode solutions in this setting and give a proof in Section~\ref{sec:NSC}.


\section{Interaction between two modes}
\label{sec:two}

In this section, we characterize two Fourier modes which do not interact (i.e., which do not give contribution through the nonlinear interaction of \eqref{cond:Euler}).
The characterization to be given in Lemma~\ref{lem:two} and Proposition~\ref{prop:two} will be a basic tool to prove Theorem~\ref{thm:main}.
This corresponds to Step~0 of the proof for the 2D result, but the condition is more complicated.

\begin{lem}\label{lem:two}
Let $\mbn_1,\mbn_2\in \R ^3\setminus \{ \mbz \}$ be two frequencies and assume they are linearly independent.
Let $\mathbf{e}^\perp$ be (one of) the unit vector perpendicular to $\mbn_1$ and $\mbn_2$, and define $\mathbf{e}^\parallel _\mbn:=\mathbf{e}^\perp \gaiseki \frac{\mbn}{|\mbn|}$ for $\mbn \in \{ \mbn_1,\mbn_2,\mbn_1+\mbn_2\}$, so that $\mathrm{Ran}\,\hat{\HP}_\mbn =\mathbb{C} \mathbf{e}^\parallel _\mbn \oplus \mathbb{C}\mathbf{e}^\perp$.
Let $\mathbf{u}_{1},\mathbf{u}_{2}\in \Bo{C}^3$ satisfy $\mathbf{u}_1 \cdot \mbn_1=\mathbf{u}_{2}\cdot \mbn_2=0$; i.e., $\hat{\HP}_{\mbn_j}\mathbf{u}_j=\mathbf{u}_j$, $j=1,2$, and write $\mathbf{u}_j =u_j^\parallel \mathbf{e}^\parallel _{\mbn_j}+u_j^\perp \mathbf{e}^\perp$.
Then, we have
\eqq{&\hat{\HP}_{\mbn_1+\mbn_2}\big[ (\mathbf{u}_1\cdot \mbn_2)\mathbf{u}_{2}+(\mathbf{u}_{2}\cdot \mbn_1)\mathbf{u}_{1}\big] \\
&\quad = \frac{(\mbn_1\gaiseki\mbn_2)\cdot \mbe^\perp}{|\mbn_1||\mbn_2|}\bigg\{ \frac{u_1^\parallel u_2^\parallel (|\mbn_2|^2-|\mbn_1|^2)}{|\mbn_1+\mbn_2|}\mathbf{e}^\parallel _{\mbn_1+\mbn_2}+\Big[ u_1^\parallel u_2^\perp |\mbn_2|-u_2^\parallel u_1^\perp |\mbn_1|\Big] \mathbf{e}^\perp \bigg\} .
}
In particular, for linearly independent $\mbn_1$ and $\mbn_2$, the output of the nonlinear interaction in \eqref{cond:Euler} between two modes $\mbu_1e^{i\mbn_1\cdot \mbx}$, $\mbu_2e^{i\mbn_2\cdot \mbx}$ (satisfying $\mathbf{u}_1 \cdot \mbn_1=\mathbf{u}_{2}\cdot \mbn_2=0$) has
\begin{alignat*}{2}
&\text{$\bullet$~non-zero component parallel to $\mbe^\parallel _{\mbn_1+\mbn_2}$} &\quad &\text{if and only if\quad $u_1^\parallel u_2^\parallel (|\mbn_1|^2-|\mbn_2|^2)\neq 0$,\quad and}\\
&\text{$\bullet$~non-zero component parallel to $\mbe^\perp$} & &\text{if and only if\quad $u_1^\parallel u_2^\perp |\mbn_2|\neq u_2^\parallel u_1^\perp |\mbn_1|$.}
\end{alignat*}
\end{lem}
\begin{proof}
A straightforward calculation using
\eqs{
(\mathbf{e}^\perp \gaiseki \mbn_1)\cdot \mbn_2 =-(\mathbf{e}^\perp \gaiseki \mbn_2)\cdot \mbn_1 =(\mbn_1\gaiseki\mbn_2)\cdot \mbe^\perp ,\\
(\mbe^\perp\gaiseki\mbn_j)\cdot \big(\mbe^\perp \gaiseki (\mbn_1+\mbn_2)\big) =\mbn_j\cdot (\mbn_1+\mbn_2) \qquad (j=1,2)
}
yields that
\eqq{
\big[ (\mathbf{u}_1\cdot \mbn_2)\mathbf{u}_{2}+(\mathbf{u}_{2}\cdot \mbn_1)\mathbf{u}_{1}\big] \cdot \mathbf{e}^\parallel_{\mbn_1+\mbn_2}
&= u^\parallel _{1}u^\parallel_{2}\Big( \big[ (\mathbf{e}^\perp \gaiseki \tfrac{\mbn_1}{|\mbn_1|})\cdot \mbn_2 \big] \big[ (\mathbf{e}^\perp \gaiseki \tfrac{\mbn_2}{|\mbn_2|})\cdot (\mathbf{e}^\perp \gaiseki \tfrac{\mbn_1+\mbn_2}{|\mbn_1+\mbn_2|})\big] \\
&\qquad\qquad +\big[ (\mathbf{e}^\perp \gaiseki \tfrac{\mbn_2}{|\mbn_2|})\cdot \mbn_1 \big] \big[ (\mathbf{e}^\perp \gaiseki \tfrac{\mbn_1}{|\mbn_1|})\cdot (\mathbf{e}^\perp \gaiseki \tfrac{\mbn_1+\mbn_2}{|\mbn_1+\mbn_2|})\big] \Big) \\
&=u^\parallel _{1}u^\parallel_{2} \tfrac{(\mbn_1\times \mbn_2)\cdot \mbe^\perp}{|\mbn_1||\mbn_2||\mbn_1+\mbn_2|}\big( |\mbn_2|^2-|\mbn_1|^2\big) ,\\
\big[ (\mathbf{u}_1\cdot \mbn_2)\mathbf{u}_{2}+(\mathbf{u}_{2}\cdot \mbn_1)\mathbf{u}_{1}\big] \cdot \mathbf{e}^\perp
&=u_1^\parallel \big[ (\mbe^\perp \gaiseki \tfrac{\mbn_1}{|\mbn_1|}) \cdot \mbn_2\big] u_2^\perp +u_2^\parallel \big[ (\mbe^\perp \gaiseki \tfrac{\mbn_2}{|\mbn_2|}) \cdot \mbn_1\big] u_1^\perp \\
&=\tfrac{(\mbn_1\times \mbn_2)\cdot \mbe^\perp}{|\mbn_1||\mbn_2|}\big( u_1^\parallel u_2^\perp |\mbn_2|-u_2^\parallel u_1^\perp |\mbn_1|\big) .
}
The claim follows.
\end{proof}

For $\mbom _1,\mbom _2\in \Stwo$ with $\mbom _2\neq \pm \mbom _1$, let $\Sc{R}_{\mbom_1\mapsto \mbom_2}$ be the (unique) rotation on $\R^3$ mapping $\mbom_1$ to $\mbom_2$ along the geodesic of $\Stwo$ connecting these points.
In other words, $\Sc{R}_{\mbom_1\mapsto \mbom_2}$ is the rotation around the axis $\mbom_1\gaiseki\mbom_2$ by the angle $\theta \in (0,\pi)$ with $\cos \theta =\mbom_1\cdot \mbom_2$.
Note that $\Sc{R}_{\mbom_1\mapsto \mbom_2}$ maps the tangent plane $T_{\mbom _1}\Stwo$ isometrically to $T_{\mbom _2}\Stwo$.
We regard $\Sc{R}_{\mbom_1\mapsto \mbom_2}$ as the operator on $\Bo{C}^3$ by letting it act on the real and the imaginary parts respectively.

\begin{prop}\label{prop:two}
Let $\mbn_1,\mbn_2\in \R^3\setminus \{ \mbz \}$ be two frequencies such that $\mbn_1\neq \pm \mbn_2$, and let $\mathbf{u}_{1},\mathbf{u}_{2}\in \Bo{C}^3\setminus \{ \mbz \}$ satisfy $\mathbf{u}_1 \cdot \mbn_1=\mathbf{u}_{2}\cdot \mbn_2=0$.
Then, we have the identity
\eq{cond1}{\hat{\HP}_{\mbn_1+\mbn_2}\big[ (\mathbf{u}_1\cdot \mbn_2)\mathbf{u}_{2}+(\mathbf{u}_{2}\cdot \mbn_1)\mathbf{u}_{1}\big] =\mbz}
if and only if one of the following holds:
\begin{enumerate}
\item $\mbn_1$ and $\mbn_2$ are linearly dependent.
(No additional condition is imposed on $\mathbf{u}_{1},\mathbf{u}_{2}$.)
\item (i) does not hold, and the real and the imaginary parts of $\mathbf{u}_1,\mathbf{u}_2$ are all perpendicular to the plane containing $\mbn_1,\mbn_2$ and the origin.
\item (i), (ii) do not hold, $|\mbn_1|=|\mbn_2|$, and moreover 
$\mathbf{u}_2=\ga \Sc{R}_{\hat{\mbn}_1\mapsto \hat{\mbn}_2}\mathbf{u}_1$
for some $\ga \in \Bo{C}\setminus \{ 0\}$, where $\hat{\mbn}_j:=\mbn_j/|\mbn_j|$.
\end{enumerate}
\end{prop}

\begin{rem}
By the above proposition, we see the remarkable fact (but similar to the 2D case) that two frequencies with no interaction must have exactly the same size, except for rather trivial situations (i), (ii).
\end{rem}

\begin{rem}\label{rem:beltrami1}
When $\mbn_1,\mbn_2$ are linearly independent, the conditions (ii), (iii) in Proposition~\ref{prop:two} can be rewritten with the notation introduced in Lemma~\ref{lem:two} as follows:
\begin{alignat*}{2}
\textrm{(ii)} ~~ &\Longleftrightarrow &\quad \textrm{(ii)'}&\quad u_1^\parallel =u_2^\parallel =0, \\
\textrm{(iii)} ~~ &\Longleftrightarrow & \textrm{(iii)'}&\quad
u_1^\parallel u_2^\parallel \neq 0,\quad |\mbn_1|=|\mbn_2|,\quad (u_2^\parallel ,u_2^\perp)=\gamma (u_1^\parallel ,u_1^\perp )\quad \text{with $\gamma =u^\parallel _2/u^\parallel _1\in \mathbb{C}\setminus \{ 0\}$.}
\end{alignat*}
Indeed, the first equivalence is trivial, while the second one can be seen by observing that $\Sc{R}_{\hat{\mbn}_1\mapsto \hat{\mbn}_2}\mbe^\parallel_{\mbn_1}=\mbe^\parallel_{\mbn_2}$ and $\Sc{R}_{\hat{\mbn}_1\mapsto \hat{\mbn}_2}\mbe^\perp =\mbe^\perp$.
\end{rem}

\begin{proof}[Proof of Proposition~\ref{prop:two}]
By Remark~\ref{rem:beltrami1}, we may consider the conditions (ii)' and (iii)' instead of (ii) and (iii), respectively.
The sufficiency of (i)--(ii)'--(iii)' for \eqref{cond1} is easily checked by the representation given in Lemma~\ref{lem:two}.
To prove the necessity of (i)--(ii)'--(iii)', assume that \eqref{cond1} holds and that (i), (ii)' do not hold.
By Lemma~\ref{lem:two}, it holds that
\[ u_1^\parallel u_2^\parallel (|\mbn_1|^2-|\mbn_2|^2)=0,\qquad u_1^\parallel u_2^\perp |\mbn_2| =u_2^\parallel u_1^\perp |\mbn_1|.\]
Since (ii)' does not hold, at least one of $u^\parallel _1$ and $u^\parallel _2$ is non-zero.
From the second equality of the above and the assumption that both $\mbu_1$ and $\mbu_2$ are non-zero, we deduce that both of $u^\parallel _1$ and $u^\parallel _2$ are non-zero.
Then, we see $|\mbn_1|=|\mbn_2|$ from the first equality, and $u^\perp _2=u^\parallel _2u^\perp _1/u^\parallel_1$ from the second one.
We have thus verified the condition (iii)'.
\end{proof}


\section{Characterization in the planar case}
\label{sec:2D}

In this section, we consider the case where $S^{conv}$ is two dimensional and conclude the following:
\begin{prop}\label{prop:planar}
Let $\mathbf{u}(t,\mbx)\in \Sc{H}_I$ be a solution of \eqref{cond:Euler} on an interval $I\subset \R$, and assume that the convex hull $S^{conv}$ of its Fourier support $S$ is a (nondegenerate) symmetric polygon on a plane $P$.
Assume further that $\mathbf{u}(t,\mbx)$ is not perpendicular to $P$ for some $(t,\mbx)\in I\times \R ^3$.
Consider the decomposition $\mathbf{u}(t,\mbx)=\mathbf{u}^\parallel (t,\mbx)+u^\perp (t,\mbx)\mbe^\perp$, where $\mathbf{u}^\parallel (t,\mbx)\in \Sc{H}_I$ is parallel to $P$ (and not identically zero), and $\mbe^\perp$ is (one of) the unit normal vector to $P$.

Then, $\mathbf{u}$ is independent of $t$, and the following holds.
\begin{enumerate}
\item The Fourier support of $\mathbf{u}^\parallel \in \Sc{H}$, denoted by $S_\parallel$, contains at least four points and is a subset of a circle centered at the origin (with its radius denoted by $\lambda$).
\item There exists a (unique) polynomial $Q$ with real coefficients and without the constant term such that $u^\perp(\mbx)=Q(\om (\mbx))-\big\langle Q(\om )\big\rangle$, where $\om (\mbx)$ is the (unique) scalar function satisfying $\nabla \gaiseki \mathbf{u}^\parallel =\lambda \om \mbe^\perp$ and $\big\langle Q(\om )\big\rangle$ is the zero mode of $Q(\om (\mbx))$.
\end{enumerate}

Conversely, any field $\mbu =\mbu ^\parallel +u^\perp \mbe^\perp \in \mathcal{H}$ satisfying the above conditions is a stationary solution of \eqref{cond:Euler}.
\end{prop}

\begin{proof}[Proof of (i)]
First, we prove the property (i) and that $\mbu^\parallel(t,\mbx)$ is independent of $t$.
Take the horizontal component (i.e., the component parallel to $P$) of the equation \eqref{cond:Euler} to obtain
\eq{cond:parallel}{\p_t\mbu^\parallel _\mbn +\frac{i}{2}\hat{\HP}_{\mbn}\sum _{\mat{\mbn_1,\mbn_2\in S_\parallel \\ \mbn_1+\mbn_2=\mbn}}\big[ (\mbu^\parallel_{\mbn_1}\cdot \mbn_2)\mbu^\parallel _{\mbn_2}+(\mbu^\parallel_{\mbn_2}\cdot \mbn_1)\mbu^\parallel _{\mbn_1}\big] =\mbz .
}
Namely, the horizontal component $\mbu^\parallel(t,\mbx)$ is in itself a solution of \eqref{cond:Euler}.
In particular, for each $\mbn\in S_\parallel$ the coefficient vector $\mbu^\parallel_{\mbn}(t)$ is non-zero for almost all $t\in I$.

We first claim that $S_\parallel$ has two linearly independent vectors.
Suppose for contradiction that it is contained in a line $\ell$ on $P$.
In this case, the second term on the left-hand side of \eqref{cond:parallel} vanishes, so the horizontal component $\mbu^\parallel$ is independent of $t$.
Choose $\mbn_1\in S_\parallel \subset \ell$ so that $|\mbn_1|=\max \{ |\mbn|:\mbn\in S_\parallel \}$, and take any $t_0\in I\setminus I_0$, where $I_0$ is defined by \eqref{def:I_0}.
Notice that the Fourier support of $u^\perp (t_0)$, which we denote by $S_\perp (t_0)$ and for which $S=S_\parallel \cup S_\perp (t_0)$ holds, contains at least one point outside $\ell$.
Choose $\mbn_2\in S_\perp (t_0)\setminus \ell$ so that $\mbn_2\cdot \mbn_1=\max \{ \mbn\cdot \mbn_1:\mbn\in S_\perp (t_0)\setminus \ell \}$, and consider the nonlinear interaction between $\mbn_1$ and $\mbn_2$.
First, we have $\mbn_1+\mbn_2\not\in \ell$ and $(\mbn_1+\mbn_2)\cdot \mbn_1>\mbn_2\cdot \mbn_1$, hence $\mbn_1+\mbn_2\not\in S$ by the definition of $\mbn_2$.
Secondly, suppose that $\mbn_1+\mbn_2=\mbn_3+\mbn_4$ for some pair of distinct points $\{ \mbn_3,\mbn_4\}\neq \{ \mbn_1,\mbn_2\}$ in $S$.
One of $\mbn_3,\mbn_4$ must be away from $\ell$, so we assume $\mbn_3\not\in \ell$.
Then, we see $\mbn_4\not\in S_\parallel$; otherwise, we would have $\mbn_3\cdot \mbn_1=\mbn_2\cdot \mbn_1+|\mbn_1|^2-\mbn_4\cdot \mbn_1>\mbn_2\cdot \mbn_1$ by the definition of $\mbn_1$ and that $\mbn_4\neq \mbn_1$, which contradicts the definition of $\mbn_2$. 
Now, the condition (ii)' in Remark~\ref{rem:beltrami1} holds for the pair $\mbn_3,\mbn_4$ and the associated coefficient vectors $\mbu_{\mbn_3}(t_0),\mbu_{\mbn_4}(t_0)$, so there is no interaction (i.e., \eqref{cond1} holds) between these frequencies.
By the equation \eqref{cond:Euler} with $\mbn =\mbn_1+\mbn_2$ and the fact $\mbn_1+\mbn_2\not\in S$, we deduce that the frequencies $\mbn_1,\mbn_2$ and the associated coefficient vectors at $t=t_0$ must also satisfy \eqref{cond1}.
But now, since $\mbn_2$ is linearly independent with $\mbn_1$, $u_{\mbn_1}^\parallel \neq 0$, $u_{\mbn_2}^\parallel =0$ and $u_{\mbn_2}^\perp (t_0)\neq 0$, Lemma~\ref{lem:two} implies that the vertical component (i.e., the component perpendicular to $P$) of the left-hand side of \eqref{cond1} is nonzero, which is a contradiction.
Therefore, $S_\parallel$ is not contained in a line, and has at least four points by symmetry.

We next show that $S_\parallel$ is contained in a circle on $P$ centered at the origin.
The proof is almost the same as that of Theorem~\ref{thm:2D} in \cite{EHS17}, but we will give it for completeness.
The proof consists of the following two steps:
\begin{enumerate}
\item[(a)] $S^{conv}_\parallel$ is inscribed in a circle centered at the origin, and $S_\parallel \cap \p S^{conv}_\parallel$ has no point other than the vertices of $S^{conv}_\parallel$.
\item[(b)] There is no point of $S_\parallel$ in the interior of $S^{conv}_\parallel$.
\end{enumerate}

To show (a), we take any side $E$ of $S^{conv}_\parallel$, and let $\mbn_1,\dots ,\mbn_p$ be the list of all points in $E\cap S_\parallel$ which are located in this order (and hence $\mbn_1,\mbn_p$ are the two endpoints of $E$).
It then suffices to verify $p=2$ and $|\mbn_1|=|\mbn_2|$.
Let $N$ be the Minkowski functional of the convex polygon $S_\parallel^{conv}$ on $P$; i.e., $N(\mbn):=\inf \{ r>0:\mbn \in rS^{conv}_\parallel \}$ for $\mbn \in P$.
Note that $S^{conv}_\parallel =\{ \mbn \in P:N(\mbn)\leq 1\}$ and $\p S^{conv}_\parallel =\{ \mbn \in P:N(\mbn)=1\}$.
Let $f$ be the linear functional on $P$ such that $f\equiv 1$ on $E$.
We see that $f\leq 1$ on $S^{conv}_\parallel$ and that $\mbn\in S^{conv}_\parallel$ and $f(\mbn )=1$ imply $\mbn\in E$. 
Note also that $N\equiv f$ on the sectorial region $\{ r\mbn :\mbn \in E,\,r\geq 0\}$.
With these functionals, consider the interaction between $\mbn_1$ and $\mbn_2$.
First, we see $\mbn_1+\mbn_2\not\in S_\parallel$ from $N(\mbn_1+\mbn_2)=f(\mbn_1+\mbn_2)=f(\mbn_1)+f(\mbn_2)=2>1$.
Secondly, if $\mbn,\mbn'\in S_\parallel$ satisfy $\mbn+\mbn'=\mbn_1+\mbn_2$, then we have $\mbn,\mbn'\in E$ (since $2=f(\mbn+\mbn')=f(\mbn )+f(\mbn')\leq 1+1=2$ implies $f(\mbn )=f(\mbn')=1$), so by the definition of $\mbn_j$'s the only possibility is that $\{ \mbn ,\mbn' \}=\{ \mbn_1,\mbn_2\}$.
Therefore, the equality \eqref{cond:parallel} with $\mbn=\mbn_1+\mbn_2$ shows that $\hat{\HP}_{\mbn_1+\mbn_2}\big[ (\mbu^\parallel _{\mbn_1}(t)\cdot \mbn_2)\mbu^\parallel _{\mbn_2}(t)+(\mbu^\parallel_{\mbn_2}(t)\cdot \mbn_1)\mbu^\parallel _{\mbn_1}(t)\big] =\mbz$ on $I$.
This and Lemma~\ref{lem:two} verify $|\mbn_1|=|\mbn_2|$, since $\mbn_1$, $\mbn_2$ are linearly independent and $\mbu^\parallel_{\mbn_1}(t),\mbu^\parallel_{\mbn_2}(t)$ are non-zero for almost all $t\in I$.
It remains to see $p=2$, so suppose $p>2$.
Then, the same argument verifies $|\mbn_{p-1}|=|\mbn_p|$, and in particular, each of the perpendicular bisectors of two segments $[\mbn_1,\mbn_2]$, $[\mbn_{p-1},\mbn_p]$ passes through the origin.
This is however impossible, because these lines are parallel.
Hence, we have $p=2$.

To prove (b), suppose for contradiction that $S_\parallel \setminus \p S_\parallel ^{conv}\neq \emptyset$, and choose $\mbn_0$ such that $N(\mbn_0)=\max \{ N(\mbn ):\mbn \in S_\parallel \setminus \p S_\parallel ^{conv}\}$, so that $0<N(\mbn_0)<1$ and $S_\parallel \setminus \p S^{conv}_\parallel \subset N(\mbn_0)S^{conv}_\parallel$.
Then, there exist two adjacent vertices $\mbn_1,\mbn_2$ of $S^{conv}_\parallel$ and $0\leq \theta <1$ such that $\mbn_0=N(\mbn_0)[(1-\theta )\mbn_1+\theta \mbn_2]$.
To derive a contradiction, we consider the interaction between $\mbn_0$ and $\mbn_2$ (note that these frequencies are linearly independent since $\theta \neq 1$).
Let $f$ be as above; namely, the linear functional on $P$ satisfying $f(\mbn_1)=f(\mbn_2)=1$, and note that $N\equiv f$ on the sectorial region $\{ r[(1-\theta )\mbn_1+\theta \mbn_2]:r\geq 0,\,0\leq \theta \leq 1\}$.
Since $f(\mbn_0+\mbn_2)=N(\mbn_0)+N(\mbn_2)>1$, we have $\mbn_0+\mbn_2\not\in S_\parallel$.
Assume that $\mbn+\mbn'=\mbn_0+\mbn_2$ for a pair $\{ \mbn,\mbn'\}\subset S_\parallel$ which is different from $\{ \mbn_0,\mbn_2\}$.
We claim that both of $\mbn,\mbn'$ are vertices of $S^{conv}_\parallel$:
In fact, if $f(\mbn)=1$, then we see from (a) that $\mbn \in \{ \mbn_1,\mbn_2\}$, and therefore $( \mbn,\mbn')=(\mbn _1, \mbn_0+(\mbn_2-\mbn_1))$.
By a simple geometric observation and $N(\mbn_0)<1$, we have $\mbn' =\mbn_0+(\mbn_2-\mbn_1)\not\in N(\mbn_0)S^{conv}_\parallel$, and thus $\mbn'$ must be one of the vertices of $S^{conv}_\parallel$.
In the same manner, if $f(\mbn')=1$, then $\mbn'=\mbn_1$ and $\mbn$ must be another vertex.
If $f(\mbn ),f(\mbn')<1$, then from $f(\mbn)+f(\mbn')=1+N(\mbn_0)$ we have $N(\mbn_0)<f(\mbn),f(\mbn')<1$, which shows that both of $\mbn,\mbn'$ are vertices (different from $\mbn_1,\mbn_2$).
Consequently, it must hold that $|\mbn|=|\mbn'|$ and, by Lemma~\ref{lem:two}, that these two frequencies do not contribute to \eqref{cond:parallel}; i.e., $\hat{\HP}_{\mbn_0+\mbn_2}\big[ (\mbu^\parallel _{\mbn}(t)\cdot \mbn')\mbu^\parallel _{\mbn'}(t)+(\mbu^\parallel_{\mbn'}(t)\cdot \mbn )\mbu^\parallel _{\mbn}(t)\big] =\mbz$ on $I$.
Using the equation \eqref{cond:parallel} at $\mbn_0+\mbn_2\not\in S_\parallel$, we deduce that $\hat{\HP}_{\mbn_0+\mbn_2}\big[ (\mbu^\parallel _{\mbn_0}(t)\cdot \mbn_2)\mbu^\parallel _{\mbn_2}(t)+(\mbu^\parallel_{\mbn_2}(t)\cdot \mbn_0)\mbu^\parallel _{\mbn_0}(t)\big] =\mbz$ on $I$.
This and Lemma~\ref{lem:two} again imply $|\mbn_0|=|\mbn_2|$, which is a contradiction.
We have thus verified that $S_\parallel \setminus \p S_\parallel ^{conv}=\emptyset$, and the claim (i).

Finally, we point out that any pair of frequencies in $S_\parallel$ is of the same distance from the origin, and hence has no contribution to the sum in \eqref{cond:parallel} by Lemma~\ref{lem:two}.
This implies that $\p_t\mbu^\parallel =\mbz$; i.e., $\mbu^\parallel$ is independent of $t$.
\end{proof}

\begin{proof}[Proof of (ii)]
Next, we prove that $u^\perp (t,\mbx)$ is also independent of $t$ and it can be represented as claimed in (ii).
Let $S_\parallel$ consist of $p$ points $\mbn_0,\mbn_1,\dots ,\mbn_{p-1}\in P$ lying on the circle in this order, and let $\lambda :=|\mbn_0|=|\mbn_1|=\cdots =|\mbn_{p-1}|$.
As in Lemma~\ref{lem:two}, the horizontal component $\mbu^\parallel(\mbx)$ can be represented as
\eq{rep:horizontal}{\mbu^\parallel (\mbx)=\sum _{j=0}^{p-1}\alpha_{j}\mbe^\parallel _{j}e^{i\mbn_j\cdot \mbx}\quad \text{with some $\alpha_{0},\dots ,\alpha_{p-1}\in \mathbb{C}\setminus \{ 0\}$, where $\mbe^\parallel _{j}:=\mbe ^\perp \gaiseki \tfrac{\mbn_j}{|\mbn_j|}$.}}
Since $\mbn_j\gaiseki \mbe^\parallel_{j}=\lambda \mbe^\perp$, the scalar function $\omega$ satisfying $\nabla \gaiseki \mathbf{u}^\parallel =\lambda \om \mbe^\perp$ is represented by
\[ \omega (\mbx) = \sum _{j=0}^{p-1}i\alpha_{j}e^{i\mbn_j\cdot \mbx}.\]
We also note that $i\alpha_{j+(p/2)}=\overline{i\alpha_j}$ for $0\leq j<p/2$, since $\mbu^\parallel$ and $\om$ are real-valued.
On the other hand, the vertical component of the equation \eqref{cond:Euler} reads as
\eq{cond:Euler^h}{\p_t u^\perp _\mbn (t)+i\sum _{\mat{(\mbn',\ti{\mbn})\in S_\parallel \times S_\perp(t) \\ \mbn'+\ti{\mbn}=\mbn}}(\mbu^\parallel_{\mbn'}\cdot \ti{\mbn})u^\perp_{\ti{\mbn}}(t) =0,\quad \text{or}\quad \p_tu^\perp +(\mbu^\parallel \cdot \nabla )u^\perp =0,\qquad t\in I,}
where $\mbu^\parallel_{\mbn'}=\alpha _j\mbe^\parallel_j$ for $\mbn'=\mbn_j$, $j=0,\dots,p-1$.
What we need to prove is that any finite-mode, real-valued and mean-zero solution $u^\perp (t,\mbx )=\sum _{\mbn}u^\perp_\mbn (t)e^{i\mbn \cdot \mbx}$ to \eqref{cond:Euler^h}
is represented as $u^\perp (t,\mbx )=u^\perp (\mbx )=Q(\om (\mbx ))-\big\langle Q(\om)\big\rangle$ for some real polynomial $Q$ without the constant term.

We continue to use the Minkowski functional of $S^{conv}_\parallel$: $N(\mbn ):=\inf \{ r>0:\mbn \in rS^{conv}_\parallel \}$ ($\mbn\in P$).
We first claim the following:
\begin{lem}\label{lem:planar1}
Assume $u^\perp \not\equiv 0$, and let $S_\perp :=\cup _{t\in I}S_\perp (t)\,(\neq \emptyset )$.
Define 
\[ q:=\max \{ N(\mbn ):\mbn \in S_\perp\} >0,\]
i.e., $q$ is the smallest number satisfying $S_\perp \subset qS^{conv}_\parallel$.
Then, $q$ must be an integer and
\eqs{\{ \mbn \in S_\perp :N(\mbn )=q\} ~=~\{ \ti{\mbn}_{j,k}: j=0,1,\dots ,p-1,~k=0,1,\dots ,q-1 \} ,\\
\ti{\mbn}_{j,k}:=(q-k)\mbn_j+k\mbn_{j+1},}
where we use the convention $\mbn_p=\mbn_0$.
In other words, the set $S_\perp \cap \p [qS^{conv}_\parallel ]$ consists of the vertices of $qS^{conv}_\parallel$ and the points that equally divide each side of $\p [qS^{conv}_\parallel ]$ into $q$ pieces.
\end{lem}

\begin{proof}
We shall show $S_\perp (t_0)\cap \p [qS^{conv}_\parallel ]=\{ \ti{\mbn}_{j,k}\}$ for any $t_0\in I$ such that $S_\perp (t_0)\cap \p [qS^{conv}_\parallel ]\neq \emptyset$ (such a time $t_0$ exists by the definition of $q$).
Pick up any $\ti{\mbn}\in S_\perp (t_0)\cap \p [qS^{conv}_\parallel ]$, then $\ti{\mbn}$ can be written as $\ti{\mbn}=q\{ (1-\th )\mbn_{j}+\th \mbn_{j+1}\}$ with $j\in \{ 0,1,\dots ,p-1\}$ and $\th \in [0,1)$ in a unique way.

Since $\mbu^\parallel_{\mbn_{j+1}}\neq 0$, and $\mbn_{j+1}$, $\ti{\mbn}$ are linearly independent, the interaction between $\mbn_{j+1}$ and $\ti{\mbn}$ gives non-zero contribution at $\mbn_{j+1}+\ti{\mbn}$; i.e., $(\mbu^\parallel_{\mbn_{j+1}}\cdot \ti{\mbn} )u^\perp _{\ti{\mbn}}(t_0)\neq 0$.
On the other hand, since $N$ coincides with a linear functional $f$ on $P$ in the sectorial region $\{ s\mbn_{j}+t\mbn_{j+1}:s,t\ge 0\}$, we have $N(\mbn_{j+1}+\ti{\mbn})=1+q>q$, and thus $\mbn_{j+1}+\ti{\mbn}\not\in S_\perp$ by the definition of $q$. 
Hence, from the equation \eqref{cond:Euler^h} there must be $(\mbn',\ti{\mbn}')\in S_\parallel \times S_\perp (t_0)$, which is different from $(\mbn_{j+1},\ti{\mbn})$, such that $\mbn'+\ti{\mbn}'=\mbn_{j+1}+\ti{\mbn}$ and $(\mbu_{\mbn'}^\parallel \cdot \ti{\mbn}')u^\perp _{\ti{\mbn}'}(t_0)\neq 0$.
The only possible one is $(\mbn',\ti{\mbn}')=(\mbn_j,\ti{\mbn}+(\mbn_{j+1}-\mbn_j))$; otherwise, $f(\mbn')<f(\mbn_{j+1})$ and then $f(\mbn'+\ti{\mbn}')=f(\mbn')+f(\ti{\mbn}')<f(\mbn_{j+1})+q=f(\mbn_{j+1}+\ti{\mbn})$, which is a contradiction.
As a consequence, the point $\ti{\mbn}+(\mbn_{j+1}-\mbn_j)$ must be in $S_\perp (t_0)$.
It then must hold that $\ti{\mbn}+r(\mbn_{j+1}-\mbn_j)=q\mbn_{j+1}$ for some positive integer $r$, since otherwise the above procedure could be repeated to create a point in $S_\perp (t_0)$ outside the polygon $qS_\parallel ^{conv}$, contradicting the definition of $q$.
Now, we have $q\mbn_{j+1}\in S_\perp (t_0)\cap \p [qS_\parallel ^{conv}]$, so setting $\ti{\mbn}=q\mbn_{j+1}$ and repeating this argument to conclude that $q$ is an integer and the complete list of the points in $S_\perp (t_0)\cap \p [qS_\parallel ^{conv}]$ is $\{ \ti{\mbn}_{j,k}\} _{0\le j\le p-1,\,0\le k\le a-1}$.
\end{proof}

Our next claim is as follows:
\begin{lem}\label{lem:planar2}
Assume $u^\perp \not\equiv 0$, and let $q>0$ be an integer defined in Lemma~\ref{lem:planar1}.
Then, there exists a (unique) constant $\be _q\in \R \setminus \{ 0\}$ such that the function $u^\perp (t,\mbx )-\be _q\big\{ \om (\mbx )^q-\big\langle \om ^q\big\rangle \big\}$ is another (real-valued, mean-zero) solution of the same equation as \eqref{cond:Euler^h} for $u^\perp$ and its Fourier support is contained in $(q-1)S_\parallel ^{conv}$.
In particular, $u^\perp (t,\mbx )=\beta _1 \om (\mbx )$ when $q=1$.
\end{lem}

\begin{proof}
From the argument in the proof of Lemma~\ref{lem:planar1}, we deduce that
\eqq{\big( \mbu^\parallel_{\mbn_{j+1}}\cdot \ti{\mbn}_{j,k}\big) u^\perp _{\ti{\mbn}_{j,k}}(t) +\big( \mbu^\parallel _{\mbn_{j}}\cdot \ti{\mbn}_{j,k+1}\big) u^\perp _{\ti{\mbn}_{j,k+1}}(t)=0,\qquad 0\!\le \!j\!\le\! p\!-\!1,\hx 0\!\le \!k\!\le\! q\!-\!1,\hx t\in I,}
where we have used the convention $\ti{\mbn}_{j,q}=\ti{\mbn}_{j+1,0}$, $\ti{\mbn}_{p,0}=\ti{\mbn}_{0,0}$.
(More precisely, this equality has been verified for $t\in I$ satisfying $S_\perp (t)\cap \p [qS_\parallel ^{conv}]\neq \emptyset$, while otherwise it holds trivially.)
Substituting $\mbu^\parallel _{\mbn_j}=\alpha_j\mbe^\parallel_j$ and noticing that $\mbe^\parallel _j \cdot \mbn_{j+1}=-\mbe^\parallel_{j+1}\cdot \mbn_j\neq 0$, we have
\eqq{u^\perp _{\ti{\mbn}_{j,k+1}}(t)=\frac{(q-k)\al _{j+1}}{(k+1)\al _j}u^\perp _{\ti{\mbn}_{j,k}}(t) ,\qquad 0\!\le \!j\!\le\! p\!-\!1,\hx 0\!\le \!k\!\le\! q\!-\!1,\hx t\in I.}
For each $t\in I$, define the complex number $\be _q(t)$ by $u^\perp _{\ti{\mbn}_{0,0}}(t)=(i\al _0)^q\be _q(t)$.
(The map $t\mapsto \be _q(t)$ is smooth, as $u^\perp_{\ti{\mbn}_{0,0}}(t)$ is smooth.)
Then, all of $u^\perp _{\ti{\mbn}_{j,k}}$ is determined by the above relation as 
\eqq{u^\perp_{\ti{\mbn}_{j,k}}(t) =\left( \begin{matrix} q \\ k \end{matrix}\right) (i\al _j)^{q-k}(i\al _{j+1})^k\be _q(t),\qquad 0\!\le \!j\!\le\! p\!-\!1,\hx 0\!\le \!k\!\le\! q\!-\!1,\hx t\in I.}
Since $u^\perp (t,\mbx )$ is real-valued, $\ti{\mbn}_{0,0}=-\ti{\mbn}_{p/2,0}$ and $i\alpha_{p/2}=\overline{i\alpha_0}$, we can show that $\beta _q(t)\in \R$:
\[ \overline{\beta_q(t)}=\overline{\Big( \frac{u^\perp _{\ti{\mbn}_{0,0}}(t)}{(i\al_0)^q}\Big)}=\frac{u^\perp_{\ti{\mbn}_{p/2,0}}(t)}{(i\al_{p/2})^q}=\beta_q(t).\]
On the other hand, from the Fourier representation of $\om$ we see that
\eqq{\om (\mbx )^q=\sum _{j=0}^{p-1}\sum _{k=0}^{q-1}\left( \begin{matrix} q \\ k \end{matrix}\right) (i\al _j)^{q-k}(i\al _{j+1})^ke^{i\ti{\mbn}_{j,k}\cdot \mbx}+\eta _q(\mbx )}
for some function $\eta _q$ on $\R^3$ whose Fourier support is finite and contained in $\{ \mbn \in P:N(\mbn )<q\}$ (the interior of $qS_\parallel ^{conv}$).
Therefore, the Fourier support of the (real-valued, mean-zero) function $\zeta (t,\mbx ):=u^\perp (t,\mbx )-\beta _q(t)\big\{ \om (\mbx )^q-\big\langle \om ^q\big\rangle \big\}$, denoted by $S'$, is also contained in $\{ \mbn :N(\mbn )<q\}$.

We next claim that $S'\subset (q-1)S_\parallel ^{conv}$.
Observe that
\eq{eq:polyom}{(\mbu^\parallel(\mbx)\cdot \nabla )\big\{ \om (\mbx )^q-\big\langle \om ^q\big\rangle \big\}=q\om (\mbx )^{q-1}\sum _{j,l=0}^{p-1}(\alpha _j\mbe^\parallel _j)\cdot (-\alpha_l\mbn_l)e^{i(\mbn_j+\mbn_l)\cdot \mbx}=0,}
since $\mbe^\parallel _j\cdot \mbn_l=-\mbe^\parallel_l\cdot \mbn_j$.
Hence, by the equation \eqref{cond:Euler^h}, $\zeta$ solves
\eq{eq:zeta}{\p _t\zeta (t,\mbx )+(\mbu^\parallel (\mbx )\cdot \nabla )\zeta (t,\mbx )=-\beta '_q(t)\big\{ \om (\mbx )^q-\big\langle \om^q\big\rangle \big\} .}
Suppose for contradiction that $\ti{q}:=\max \{ N(\mbn ):\mbn \in S'\} >q-1$.
Then, noticing that the Fourier support of $\p _t\zeta (t)$ and that of the right-hand side of \eqref{eq:zeta} are contained in $\{ \mbn :N(\mbn )\leq q\}$, the same argument as for Lemma~\ref{lem:planar1} would imply that $\ti{q}$ is an integer, contradicting $\ti{q}<q$ which we have shown above.
As a consequence, in the case $q=1$ we have $\ze \equiv 0$ since $\zeta$ is mean-zero.

Now, we have only to show that $\be _q(t)$ is independent of $t\in I$ (which implies $\beta _q\neq 0$).
When $q=1$, we immediately obtain $\beta'_q(t)=0$ from the equation \eqref{eq:zeta}.
We assume $q>1$ and compare the Fourier coefficient of both sides of \eqref{eq:zeta} at $\ti{\mbn}_{0,0}$:
\eqq{i\sum _{\mat{(\mbn',\ti{\mbn})\in S_\parallel \times S'\\ \mbn'+\ti{\mbn}=\ti{\mbn}_{0,0}}}(\mbu^\parallel _{\mbn'}\cdot \ti{\mbn})\zeta _{\ti{\mbn}}(t)=-(i\al _0)^q\beta '_q(t).}
Observe that the only possible pair $(\mbn',\ti{\mbn})\in S_\parallel \times S'$ satisfying $\mbn'+\ti{\mbn}=\ti{\mbn}_{0,0}=q\mbn_0$ is $(\mbn',\ti{\mbn})=(\mbn_0,(q-1)\mbn_0)$:
In fact, noticing $S'\subset (q-1)S_\parallel ^{conv}$, this can be verified by a simple argument using a linear functional $g$ on $P$ satisfying $g(\mbn_0)=1$ and $g<1$ on $S^{conv}_\parallel \setminus \{ \mbn_0\}$.
Then, since these frequencies do not interact (as they are linearly dependent), the left-hand side of the above equality is zero, and so $\be '_q(t)=0$.
\end{proof}

By applying Lemmas \ref{lem:planar1} and \ref{lem:planar2} repeatedly (at most $q$ times), we obtain the (unique) polynomial $Q$ of degree $q$ with real coefficients and without the constant term such that  $u^\perp (t,\mbx )=Q(\om (\mbx ))-\big\langle Q(\om ) \big\rangle$.
We have thus proved (ii) and that $u^\perp$ is independent of $t$.
\end{proof}

\begin{proof}[Proof of the converse]
We note that any flow $\mbu^\parallel (\mbx )\in \mathcal{H}$ with Fourier modes on $P$ and parallel to $P$ everywhere is represented as \eqref{rep:horizontal}.
Since the Fourier modes of $\mbu^\parallel$ are equidistant from the origin, we see by Lemma~\ref{lem:two} that $\mbu^\parallel$ satisfies the equation \eqref{cond:parallel}.
It then suffices to prove that $u^\perp (\mbx ):=Q(\om (\mbx ))-\big\langle Q(\om )\big\rangle$ solves the equation \eqref{cond:Euler^h}, which follows from the calculation \eqref{eq:polyom}.

This concludes the proof of Proposition~\ref{prop:planar}.
\end{proof}


\section{Characterization in the 3D case}\label{sec:S}

In this section, we consider the case where the Fourier support $S$ has three linearly independent vectors.
By exploiting the characterization of two non-interacting frequencies given in Proposition~\ref{prop:two}, we shall prove:
\begin{thm}\label{thm:main2}
Any solution $\mathbf{u}\in \Sc{H}_I$ of \eqref{cond:Euler} is a (stationary) Beltrami flow when its Fourier support is not contained in a plane.
\end{thm}

We introduce some terminology to be used frequently in the proof of Theorem~\ref{thm:main2}:
\begin{defn}
(i) Let $S$ be a finite subset of $\R^3\setminus \{ \mbz \}$.
We call a pair of two distinct points $\mbn_1,\mbn_2\in S$ \emph{simply interacting pair} in $S$ (SIP for short) if the following conditions hold:
\begin{itemize}
\item $\mbn_1+\mbn_2\not\in S$.
\item If two distinct points $\mbn_3,\mbn_4\in S$ satisfy $\mbn_1+\mbn_2=\mbn_3+\mbn_4$, then $\{ \mbn_1,\mbn_2\}=\{ \mbn_3,\mbn_4\}$.
\end{itemize} 

(ii) Let $\mbn\in \R ^3\setminus \{ \mbz \}$.
We call a vector $\mathbf{u}_\mbn\in \Bo{C}^3\setminus \{ \mbz \}$ \emph{positive (resp. negative) Beltrami vector} at $\mbn$ ($BV^\pm$ for short) if it is an eigenvector of $i\mbn \times$ with respect to the eigenvalue $+|\mbn |$ (resp. $-|\mbn |$), or equivalently (by Lemma~\ref{lem:Beltrami}), if the following conditions hold:
\begin{itemize}
\item $\Re \mathbf{u}_\mbn \cdot \mbn=\Im \mathbf{u}_\mbn \cdot \mbn=\Re \mathbf{u}_\mbn \cdot \Im \mathbf{u}_\mbn=0$,\hx $|\Re \mathbf{u}_\mbn|=|\Im \mathbf{u}_\mbn|$.
\item $(\mbn,\Re \mathbf{u}_\mbn,\Im \mathbf{u}_\mbn)$ is a right-handed (resp. left-handed) system.
\end{itemize}
\end{defn}

\begin{rem}\label{rem:SIP}
(i) When $\mathbf{u}(t,\mbx)=\sum _{\mbn\in S}\mathbf{u}_\mbn(t)e^{i\mbn\cdot \mbx}\in \Sc{H}_I$ is a solution to \eqref{Euler} on $I\times \R ^3$, the coefficient vectors $\mathbf{u}_1=\mathbf{u}_{\mbn_1}$, $\mathbf{u}_2=\mathbf{u}_{\mbn_2}$ at an SIP of two frequencies $\mbn_1,\mbn_2\in S$ always satisfy \eqref{cond1}, due to \eqref{cond:Euler} with $\mbn=\mbn_1+\mbn_2$.
Note that the converse is not necessarily true; namely, \eqref{cond1} may be true even for non-SIP frequencies.

(ii) We note that SIP is not a transitive relation.
For instance, when $S$ contains four points $\mbn_1,\mbn_2,\mbn_3,\mbn_4$ that form a parallelogram in this order (i.e., $\mbn_1+\mbn_3=\mbn_2+\mbn_4$), the pairs $(\mbn_1,\mbn_2)$, $(\mbn_2,\mbn_3)$, $(\mbn_3,\mbn_4)$, $(\mbn_4,\mbn_1)$ can be SIP but not are $(\mbn_1,\mbn_3)$ and $(\mbn_2,\mbn_4)$.

(iii)
From Lemma~\ref{lem:Beltrami}, $\mathbf{u}(\mbx)=\sum _{\mbn\in S}\mathbf{u}_\mbn e^{i\mbn\cdot \mbx}\in \Sc{H}$ is a Beltrami flow if and only if $S\subset \lambda \Stwo$ for some $\lambda >0$ and the coefficient vectors are all $BV^+$ or all $BV^-$ (corresponding to the eigenvalue $\la$ or $-\la$, respectively).

(iv) It is easy to see that for any $\mbn \in \mathbb{R}^3\setminus \{ \mbz \}$, each of $BV^+$ and $BV^-$ at $\mbn$ is invariant under multiplication by non-zero complex number.
Moreover, when $\mbn_1,\mbn_2\in \R^3\setminus \{ \mbz \}$ are linearly independent and $|\mbn_1|=|\mbn_2|$, a vector $\mathbf{v}\in \Bo{C}^3$ is $BV^+$ (resp. $BV^-$) at $\mbn_1$ if and only if $\Sc{R}_{\hat{\mbn}_1\mapsto \hat{\mbn}_2}\mathbf{v}$ is $BV^+$ (resp. $BV^-$) at $\mbn_2$, because the geometric conditions determining $BV^\pm$ are not disrupted by the rotation $\Sc{R}_{\hat{\mbn}_1\mapsto \hat{\mbn}_2}$.
Consequently, in the situation of Proposition~\ref{prop:two}~(iii), if $\mathbf{u}_1$ is shown to be $BV^+$ (resp. $BV^-$) at $\mbn_1$, then $\mathbf{u}_2$ is also $BV^+$ (resp. $BV^-$) at $\mbn_2$.
\end{rem}

Let us begin to prove Theorem~\ref{thm:main2}.
We first give an analog of Step~1 in the 2D case:
\begin{prop}\label{prop:S1}
Let $\mathbf{u}=\sum _{\mbn\in S}\mathbf{u}_\mbn(t)e^{i\mbn\cdot \mbx}\in \Sc{H}_I$ be a solution of \eqref{cond:Euler} on $I\subset \R$ and $t_0\in I\setminus I_0$, where $I_0$ is defined in \eqref{def:I_0}.
Assume that $S$ is not contained in a plane.
Then, the following properties hold:
\begin{enumerate}
\item The polyhedron $S^{conv}$ (convex hull of $S$) is inscribed in a sphere centered at the origin.
\item There is no point of $S$ on each edge of $S^{conv}$ except for two endpoints.
\item For each edge of $S^{conv}$, its two endpoints are SIP in $S$, and (iii) of Proposition~\ref{prop:two} holds for these vertices and the associated coefficient vectors at $t=t_0$. 
\item For each vertex $\mbn$ of $S^{conv}$, $\Re \mathbf{u}_\mbn(t_0), \Im \mathbf{u}_\mbn(t_0) \in \R ^3$ are linearly independent.
\end{enumerate}
\end{prop}

Before proving the above proposition, we prepare two lemmas.
\begin{lem}\label{lem:3-1}
Under the assumptions in Proposition~\ref{prop:S1}, let $E$ be an arbitrary edge of $S^{conv}$.
\begin{enumerate}
\item If $E\cap S$ consists of the endpoints of $E$, then the endpoints are SIP in $S$ and the associated coefficient vectors at $t=t_0$ satisfy either (ii) or (iii) of Proposition~\ref{prop:two}.
\item If $E\cap S$ has more than two points, then (ii) of Proposition~\ref{prop:two} occurs for any pair of points in $E\cap S$ and the associated coefficient vectors at $t=t_0$.
\end{enumerate}
In particular, any pair of adjacent vertices of $S^{conv}$ do not interact.
\end{lem}

\begin{proof}
Let $E\cap S$ consist of $p\,(\ge 2)$ points $\mbn_1,\mbn_2,\dots ,\mbn_p$ located in this order (thus the endpoints $\mbn_1,\mbn_p$ are adjacent vertices of $S^{conv}$).

(i) 
When $p=2$, it is easily shown that two endpoints $\mbn_1,\mbn_2$ are SIP.
To see this, let $F_1,F_2$ be the two faces of $S^{conv}$ sharing $E$ as a side, $f_1,f_2$ be the linear functionals on $\R^3$ which are identically equal to $1$ on $F_1$ and $F_2$, respectively, and define $f:=(f_1+f_2)/2$.
Then, we see that $S^{conv}\subset \{ \mbn\in \R ^3:f(\mbn)\le 1\}$, and that $S^{conv}\cap \{ f(\mbn)=1\}=E$.
First, $f(\mbn_1+\mbn_2)=f(\mbn_1)+f(\mbn_2)=2>1$ and thus $\mbn_1+\mbn_2\not\in S$.
Secondly, assume that $\mbn',\mbn''\in S$ satisfy $\mbn_1+\mbn_2=\mbn'+\mbn''$, then it holds that $f(\mbn')+f(\mbn'')=2$ and $f(\mbn'),f(\mbn'')\le 1$, which implies that $f(\mbn')=f(\mbn'')=1$ and thus $\mbn',\mbn''\in E$.
But $S\cap E=\{ \mbn_1,\mbn_2\}$, so we have $\{ \mbn_1,\mbn_2\}=\{ \mbn',\mbn''\}$.
This shows that $\mbn_1,\mbn_2$ are SIP in $S$.
By Remark~\ref{rem:SIP} (i), the coefficient vectors at $\mbn_1,\mbn_2$ satisfies \eqref{cond1}, and hence the claim follows from Proposition~\ref{prop:two} (since $\mbn_1,\mbn_2$ are linearly independent).

(ii)
We next assume $p\ge 3$.
It follows that $\mbn_1,\mbn_2$ are SIP:
In fact, $\mbn_1+\mbn_2=\mbn'+\mbn''$, $\mbn',\mbn''\in S$ imply that $\mbn',\mbn''\in E$ (by the same argument as the above case of $p=2$) and that the two segments $[\mbn_1,\mbn_2]$ and $[\mbn',\mbn'']$ have the common middle point.
By the definition of $\mbn_1,\mbn_2$, it must hold that $\{ \mbn_1,\mbn_2\}=\{ \mbn',\mbn''\}$.
A similar argument shows that $\mbn_1$, $\mbn_3$ are also SIP (because it is only $\mbn_2$ that is in $S\cap E$ and between $\mbn_1,\mbn_3$).
Hence, \eqref{cond1} holds for the pair $\mbn_1,\mbn_2$ and the associated coefficient vectors, so that either (ii) or (iii) of Proposition~\ref{prop:two} occurs, and the same is true for the pair $\mbn_1,\mbn_3$. 
Now, suppose $\mbn_1,\mbn_2$ are as in (iii) of Proposition~\ref{prop:two}, then the coefficient vector at $\mbn_1$ is not perpendicular to the plane containing $E$ and the origin, which shows that $\mbn_1,\mbn_3$ are also as in (iii).
This implies $|\mbn_1|=|\mbn_2|=|\mbn_3|$, which is, however, impossible because these three points are collinear.
Therefore, the coefficient vectors at $\mbn_1,\mbn_2,\mbn_3$ are all perpendicular to that plane and any two of them are as in (ii) of Proposition~\ref{prop:two}, verifying the claim for $p=3$.

In the case $p\ge 4$, we can show by induction that all of $\mbn_1,\dots ,\mbn_p$ are actually in the same relation (and hence the claim follows): 
Suppose any two of $\mbn_1,\dots ,\mbn_{q-1}$ ($4\le q\le p$) are as in (ii), and consider the nonlinear interaction contributing to $\mbn_1+\mbn_q$.
The above argument with the linear functional $f$ shows that $\mbn_1+\mbn_q\not\in S$, and that $\mbn_1+\mbn_q=\mbn'+\mbn''$, $\mbn',\mbn''\in S$, $\{ \mbn_1,\mbn_q\}\neq \{ \mbn',\mbn''\}$ imply $\mbn',\mbn''\in \{ \mbn_2,\cdots ,\mbn_{q-1}\}$.
By the induction assumption and Proposition~\ref{prop:two}, such a pair $(\mbn',\mbn'')$ is not interacting (i.e., \eqref{cond1} holds).
Then, the equation \eqref{cond:Euler} shows that \eqref{cond1} also holds for the pair $(\mbn_1,\mbn_q)$.
From Proposition~\ref{prop:two} again, this pair is also as in (ii).
In particular, the coefficient vectors at $\mbn_1,\dots ,\mbn_q$ are all perpendicular to the plane containing $E$ and the origin, and any two of them are as in (ii).
\end{proof}

\begin{lem}\label{lem:3-2}
Let $p\ge 3$ and $\mbom _1,\mbom _2,\dots ,\mbom _p$ be $p$ points located on a circle $\Sc{C}\subset \Stwo$ in this order, and assume that $\Sc{C}$ is not a great circle of $\Stwo$.
Let $A(F^*)\in (0,2\pi )$ be the area of the spherical $p$-polygon $F^*$ corresponding to the $p$-polygon $\mbom _1\mbom _2\dots \mbom _p$ (i.e., $F^*$ is the subset of $\Stwo$ enclosed by the geodesics connecting $\mbom _j$ and $\mbom _{j+1}$, $j=1,2,\dots ,p$, with the convention $\mbom _{p+1}=\mbom _1$).

Then, the operator $\Sc{R}_1:=\Sc{R}_{\mbom_p\mapsto \mbom_1}\circ \Sc{R}_{\mbom_{p-1}\mapsto \mbom_p}\circ \cdots \circ \Sc{R}_{\mbom_1\mapsto \mbom_2}$ is the rotation around the axis $\mbom_1$ by the angle either $A(F^*)$ or $-A(F^*)$.
\end{lem}

\begin{proof}
We may assume without loss of generality that the $p$-polygon $\mbom _1\mbom _2\dots \mbom _p$ is in northern hemisphere $\{ x^3>0\}$ and parallel to the $(x^1,x^2)$ plane, and that the ordering $\mbom _1,\mbom _2,\dots ,\mbom _p$ is ``westward''.
Define the orthonormal frame $\{ \mbe^1(\mbom _j),\mbe^2(\mbom _j)\}$ of the tangent plane $T_{\mbom _j}\Stwo$ by the ``eastward'' and ``northward'' unit vectors; namely,
\eqq{\mbe^1(\mbom _j):=\frac{\mbe^3\gaiseki \mbom _j}{|\mbe^3\gaiseki \mbom _j|},\qquad \mbe^2(\mbom _j):=\mbom _j\gaiseki \frac{\mbe^3\gaiseki \mbom _j}{|\mbe^3\gaiseki \mbom _j|}\qquad (j=1,2,\dots ,p),}
where $\mbe^3=(0,0,1)$ denotes the north pole.

By a simple geometric observation (see Figure~\ref{fig:rotation}), it turns out that the operator $\Sc{R}_{\mbom _j\mapsto \mbom _{j+1}}$ maps $T_{\mbom_j}\Stwo$ onto $T_{\mbom_{j+1}}\Stwo$ and acts as the rotation by the angle $2\th _j$:
\eqq{\Sc{R}_{\mbom _j\mapsto \mbom _{j+1}}\mbe^1(\mbom _j)&=(\cos 2\th _j)\mbe^1(\mbom _{j+1})+(\sin 2\th _j)\mbe^2(\mbom _{j+1}),\\
\Sc{R}_{\mbom _j\mapsto \mbom _{j+1}}\mbe^2(\mbom _j)&=-(\sin 2\th _j)\mbe^1(\mbom _{j+1})+(\cos 2\th _j)\mbe^2(\mbom _{j+1})\qquad (j=1,2,\dots ,p),}
where $\th _j$ denotes the angle between the latitude circle and the geodesic curve passing through $\mbom _j,\mbom _{j+1}$ such that $\th _j\in (0,\pi /2)$ [resp. $\in (\pi /2,\pi)$] when $\angle \mbom _j\mbom _{j-1}\mbom _{j+1}<\pi /2$ [resp. $>\pi /2$], with the convention $\mbom _{p+1}=\mbom _1$, $\mbom _0=\mbom _p$.
Therefore, the composition operator $\Sc{R}_1$ restricted to $T_{\mbom_1}\Stwo$ is the rotation on $T_{\mbom _1}\Stwo$ by $2\sum _{j=1}^p\th _j$.
It is clear that $\mbom_1$ is invariant under $\mathcal{R}_1$.

\begin{figure}\centering\includegraphics{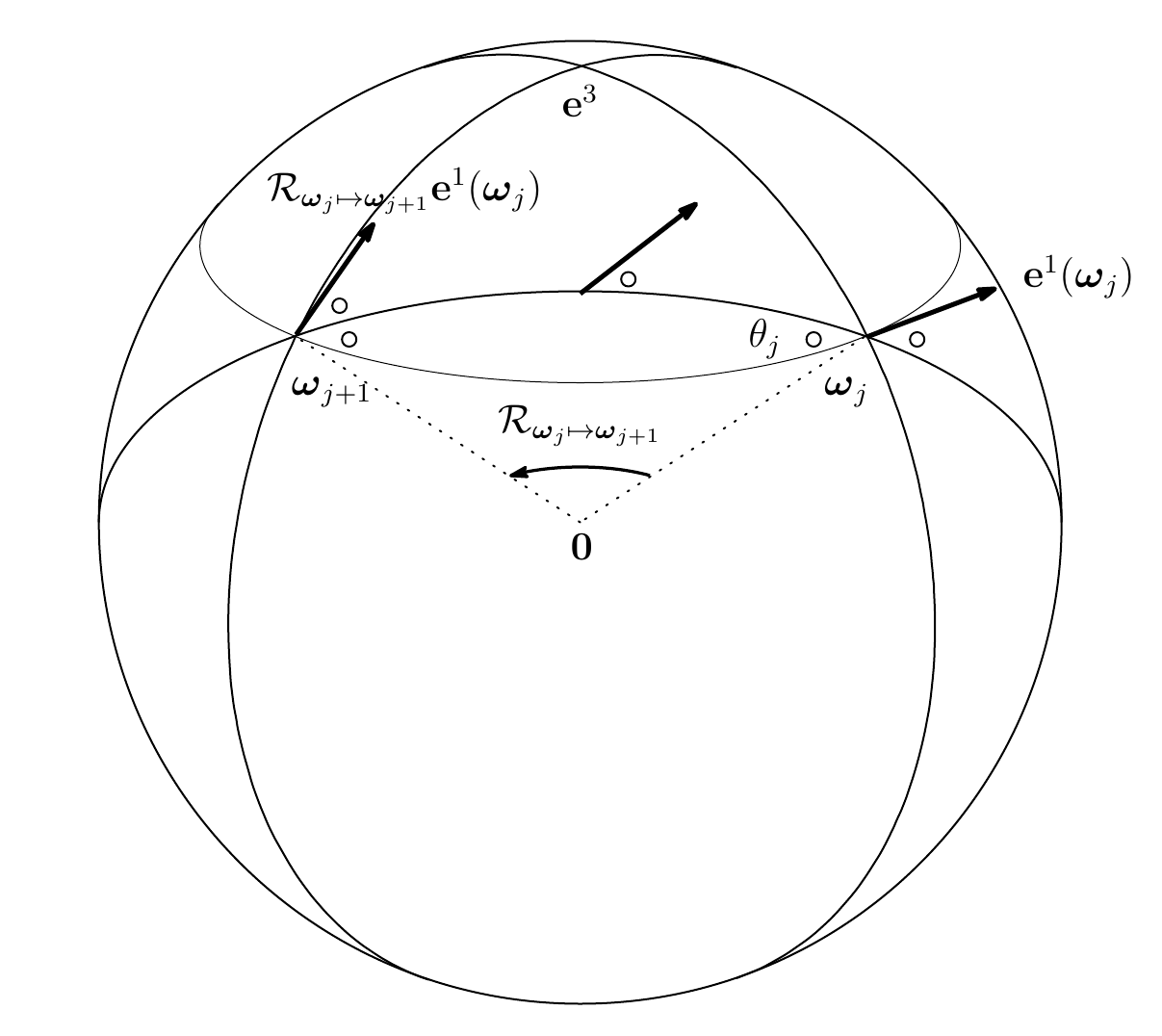}\caption{$\Sc{R}_{\mbom _j\mapsto \mbom _{j+1}}:T_{\mbom_j}\Stwo \to T_{\mbom_{j+1}}\Stwo$ is the rotation by the angle $2\th _j$.}\label{fig:rotation}\end{figure}

Now, let $\Theta (F^*)$ be the sum of the interior angles of the spherical polygon $F^*$, namely $\Theta (F^*)=p\pi -2\sum _{j=1}^p\th _j$.
The celebrated Gauss-Bonnet theorem tells us that $A(F^*)=\Theta (F^*)-(p-2)\pi$, and hence $2\sum _{j=1}^p\th _j=2\pi -A(F^*)$.
Therefore, in this case the rotation angle of $\Sc{R}_1$ is equal to $-A(F^*)$.
When the ordering $\mbom _1,\mbom _2,\dots ,\mbom _p$ is ``eastward'', the same argument concludes that the rotation angle is $A(F^*)$.
\end{proof}

\begin{proof}[Proof of Proposition~\ref{prop:S1}]
The main part of the proof is to establish (iv).
Suppose, for the sake of contradiction, that $\Re \mathbf{u}_\mbn(t_0)$ and $\Im \mathbf{u}_\mbn(t_0)$ are linearly dependent (including the case where one of them is zero; since $t_0\not\in I_0$, at least one of them is nonzero) for some vertex $\mbn$ of $S^{conv}$.
Observe that this property ``propagates'' to adjacent vertices of $S^{conv}$, because by Lemma~\ref{lem:3-1} any pair of adjacent vertices satisfies either (ii) or (iii) of Proposition~\ref{prop:two}.
Hence, for every vertex $\mbn$ of $S^{conv}$ a line $\ell (\mbn)\subset \R ^3$ passing through the origin is uniquely determined so that both $\Re \mathbf{u}_\mbn(t_0)$ and $\Im \mathbf{u}_\mbn(t_0)$ are parallel to $\ell (\mbn)$ (or one of them is zero and the other is parallel to $\ell (\mbn)$).

We claim that $\ell (\mbn_2)=\Sc{R}_{\hat{\mbn}_1\mapsto \hat{\mbn}_2}\ell (\mbn_1)$ for each pair $(\mbn_1,\mbn_2)$ of adjacent vertices of $S^{conv}$, where $\hat{\mbn}_j=\mbn_j/|\mbn_j|$.
In fact, if $\mbn_1,\mbn_2$ are as in (ii) of Proposition~\ref{prop:two}, $\ell (\mbn_1)=\ell (\mbn_2)$ is the line perpendicular to the plane including $\mbn_1$, $\mbn_2$ and the origin, and also it is the rotation axis of $\Sc{R}_{\hat{\mbn}_1\mapsto \hat{\mbn}_2}$ and thus invariant under that transformation.
For the situation (iii) of Proposition~\ref{prop:two}, this follows immediately from Proposition~\ref{prop:two}.
As a consequence, for any closed chain of edges of $S^{conv}$, say $\mbn_1\mbn_2\dots \mbn_p\mbn_1$ (where $\mbn_j$ is a vertex of $S^{conv}$ and adjacent to $\mbn_{j-1}$ for each $j$, with $\mbn_1$ adjacent to $\mbn_p$), the composition operator $\Sc{R}_1=\Sc{R}_{\hat{\mbn}_p\mapsto \hat{\mbn}_1}\circ \Sc{R}_{\hat{\mbn}_{p-1}\mapsto \hat{\mbn}_p}\circ \cdots \circ \Sc{R}_{\hat{\mbn}_1\mapsto \hat{\mbn}_2}$ does not change the initial line $\ell (\mbn_1)$.
Therefore, to derive a contradiction, it suffices to find a closed chain of edges of $S^{conv}$ such that $\mathcal{R}_1\ell (\mbn_1)\neq \ell (\mbn_1)$.

Let $\hat{S}^{conv}$ be the convex hull of $\{ \hat{\mbn}:=\mbn/|\mbn|:\text{$\mbn$ is a vertex of $S^{conv}$}\}$, the set of normalized vertices of $S^{conv}$.
Note that $\hat{S}^{conv}$ is a symmetric polyhedron inscribed in $\Stwo$.
In particular, it has at least six faces, and hence there exists a face $\hat{F}$ of $\hat{S}^{conv}$ such that the area $A(\hat{F}^*)$ of the corresponding spherical polygon $\hat{F}^*\subset \Stwo$ is not greater than $2\pi /3$.
%
%
Now, let $\hat{\mbn}_1,\hat{\mbn}_2,\dots ,\hat{\mbn}_p$ be the list of all vertices of $\hat{F}$ located in this order, and let $\mbn_1,\mbn_2,\dots ,\mbn_p$ be the corresponding vertices of $S^{conv}$.
From Lemma~\ref{lem:3-2}, we see that the resulting operator $\mathcal{R}_1$ is the rotation by the angle $\pm A(\hat{F}^*)$ around the axis $\mbn_1$.
Since $\ell (\mbn_1)$ is perpendicular to $\mbn_1$ and $\pm A(\hat{F}^*)\not\in \pi \mathbb{Z}$, we have $\mathcal{R}_1\ell (\mbn_1)\neq \ell (\mbn_1)$ and reach a contradiction, as desired.

We have thus proved (iv).
Note that the property (iv) prohibits the situation (ii) of Proposition~\ref{prop:two} for any pair of adjacent vertices of $S^{conv}$.
Then, Lemma~\ref{lem:3-1} implies (ii) and (iii).
Finally, the property (i) follows from (iii). 
\end{proof}

We are now in a position to show that:
\begin{prop}\label{prop:BV}
Under the same assumptions as in Proposition~\ref{prop:S1}, 
it holds that the coefficient vectors $\mathbf{u}_{\mbn}(t_0)$ at vertices of $S^{conv}$ are all $BV^+$ or all $BV^-$.
\end{prop}

\begin{proof}
As observed in the proof of Proposition~\ref{prop:S1}, there is a face $F$ of $S^{conv}$ such that the spherical polygon $\hat{F}^*\subset \Stwo$ corresponding to the normalized face $\hat{F}$ of $\hat{S}^{conv}$ has the area $A(\hat{F}^*)\in (0,2\pi /3]$.
Let $\mbn_1,\mbn_2,\dots ,\mbn_p\in S$ be the list of all vertices of $F$ located in this order. 

By Proposition~\ref{prop:S1}~(iii) and Proposition~\ref{prop:two}, there exist $\ga _1,\cdots ,\ga _p\in \Bo{C}\setminus \{ 0\}$ such that
\eqq{\mathbf{u}_{\mbn_1}(t_0)&=\big[ (\ga _p\Sc{R}_{\hat{\mbn}_p\mapsto \hat{\mbn}_1})\circ (\ga _{p-1}\Sc{R}_{\hat{\mbn}_{p-1}\mapsto \hat{\mbn}_p})\circ \cdots \circ (\ga _1\Sc{R}_{\hat{\mbn}_1\mapsto \hat{\mbn}_2})\big] \mathbf{u}_{\mbn_1}(t_0)\\
&=(\ga _1\ga _2\cdots \ga _p)\Sc{R}_1\mathbf{u}_{n_1}(t_0)}
with $\Sc{R}_1=\Sc{R}_{\hat{\mbn}_p\mapsto \hat{\mbn}_1}\circ \Sc{R}_{\hat{\mbn}_{p-1}\mapsto \hat{\mbn}_p}\circ \cdots \circ \Sc{R}_{\hat{\mbn}_1\mapsto \hat{\mbn}_2}$; namely, $\mathbf{u}_{\mbn_1}(t_0)\in \Bo{C}^3\setminus \{ \mbz \}$ is an eigenvector of the rotation operator $\Sc{R}_1$.
By Lemma~\ref{lem:3-2}, the rotation angle is $\pm A(\hat{F}^*)$, which is not an integral multiple of $\pi$.
We may focus on the case of $+A(\hat{F}^*)$, then it is easily verified that $\Sc{R}_1$ has three distinct eigenvalues $1,e^{iA(\hat{F}^*)},e^{-iA(\hat{F}^*)}$, with the corresponding eigenspaces $\{ \al \hat{\mbn}_1:\al \in \Bo{C}\}$, $\{ \text{$BV^+$ at $\mbn_1$}\}\cup \{ \mbz\}$ and $\{ \text{$BV^-$ at $\mbn_1$}\}\cup \{ \mbz \}$, respectively.
Since $\mathbf{u}_{\mbn_1}(t_0)\cdot \mbn_1=0$, we see that $\mathbf{u}_{\mbn_1}(t_0)$ is either $BV^+$ or $BV^-$ at $\mbn_1$.
Now, Remark~\ref{rem:SIP} (iv) together with Proposition~\ref{prop:S1} (iii) leads to the conclusion.
\end{proof}

If $\mbn_1,\mbn_2\in \R^3$ are two distinct frequencies satisfying $|\mbn_1|=|\mbn_2|$, and if $\mathbf{u}_{\mbn_1},\mathbf{u}_{\mbn_2}\in \Bo{C}^3$ are both $BV^+$ or both $BV^-$ at these frequencies, then these two modes satisfy the condition (i) or (iii) of Proposition~\ref{prop:two}, and hence do not interact.
In particular, we see from Proposition~\ref{prop:BV} that there is no contribution from the nonlinear interaction between any pair of vertices of $S^{conv}$ (not necessarily adjacent).
This property (corresponding to Step~2 in 2D) will play a key role in verifying the following proposition (corresponding to Step~3 in 2D):

\begin{prop}\label{prop:interior}
Under the same assumptions as in Proposition~\ref{prop:S1}, 
there is no point of $S$ other than the vertices of $S^{conv}$.
\end{prop}

\begin{proof}
The idea is very similar to the 2D case.
Arguing by contradiction, we suppose that $S\setminus V(S)\neq \emptyset$, where $V(S)$ denotes the set of all vertices of $S^{conv}$.

Let $N$ be the Minkowski functional of the convex set $S^{conv}$ containing the origin; i.e., 
\eqq{N(\mbn):=\inf \{ r>0:\mbn \in rS^{conv}\} .}
Note that $N$ is a norm on $\R^3$ and $S^{conv}=\{ \mbn:N(\mbn)\le 1\}$, $\p S^{conv}=\{ \mbn:N(\mbn)=1\}$.
Then, we choose $\mbn_0\in S\setminus V(S)$ so that $N(\mbn_0)=\max \{ N(\mbn):\mbn\in S\setminus V(S)\}$.
Since $\mbn_0\neq \mbz$, we have $0<N(\mbn_0)\le 1$.
Let $F$ be (one of) the face of $S^{conv}$ on which $\mbn_0/N(\mbn_0) \in \p S^{conv}$ is located, let $\ti{F}:=N(\mbn_0)F$ (so that $\mbn_0\in \ti{F}$), and denote by $\ti{P}$ the plane containing $\ti{F}$.
(Note that $\ti{F}=F$ when $N(\mbn_0)=1$.)
We also find the (unique) linear functional $f$ on $\R^3$ satisfying $f\equiv 1$ on $F$.
Observe that $f$ coincides with $N$ in the conic region $\{ r\mbn:\mbn\in F,~r\geq 0\}$, that $S^{conv}\subset \{ \mbn:f(\mbn)\le 1\}$, and that $S^{conv}\cap \{ \mbn:f(\mbn)=1\} =F$.
Moreover, by the definition of $\mbn_0$, $\mbn\in S\setminus V(S)$ implies either $f(\mbn)<f(\mbn_0)$ or $\mbn\in \ti{F}$.

Next, fix an arbitrary point $\ti{\mbn}_*\in \mathrm{Int}(\ti{F})$ and define the function $\ti{N}$ on $\ti{P}$ by
\eqq{\ti{N}(\mbn):=\inf \{ r>0: \mbn-\ti{\mbn}_*\in r[\ti{F}-\ti{\mbn}_*]\} ,\qquad \mbn\in \ti{P}.}
Note that $\ti{F}=\{ \mbn\in \ti{P}:\ti{N}(\mbn)\le 1\}$, $\p \ti{F}=\{ \mbn\in \ti{P}:\ti{N}(\mbn)=1\}$.
By replacing $\mbn_0$ if necessary, we may assume that $\ti{N}(\mbn_0)=\max \{ \ti{N}(\mbn):\mbn\in [S\setminus V(S)]\cap \ti{F}\}$.
We have $0\le \ti{N}(\mbn_0)\le 1$, and $\ti{N}(\mbn_0)<1$ if $N(\mbn_0)=1$, because $S\cap \p F\subset V(S)$ by means of Proposition~\ref{prop:S1} (ii).
Let 
\eqq{\ti{F}':=\ti{\mbn}_*+\ti{N}(\mbn_0)[\ti{F}-\ti{\mbn}_*]}
be the polygon on $\ti{P}$ obtained by contracting $\ti{F}$ by the ratio of $\ti{N}(\mbn_0)$ with respect to the base point $\ti{\mbn}_*$.
By the definition, we see $[S\setminus V(S)]\cap \ti{F}\subset \ti{F}'$, and in particular, that $\mbn\in S\setminus V(S)$ implies either $f(\mbn)<f(\mbn_0)$ or $\mbn\in \ti{F}'$.

Then, noticing $\mbn_0\in \p \ti{F}'$, define $\ti{E}'$ as (one of) the side of $\ti{F}'$ on which $\mbn_0$ is located, and denote the corresponding side of $F$ by $E$.
We can find a linear functional $g$ on $\R^3$ satisfying $g\equiv 0$ on $E$ and $g>0$ on $F\setminus E$. 
It then follows from a simple geometric observation that $g\equiv g(\mbn_0)\ge 0$ on $\ti{E}'$ and $g>g(\mbn_0)$ on $\ti{F}'\setminus \ti{E}'$.
Finally, choose $\mbn_1\in V(S)$ so that $\mbn_1$ is an endpoint of $E$ and linearly independent with $\mbn_0$.

Let us consider the nonlinear interaction contributing to the mode $\mbn_0+\mbn_1$.
By Proposition~\ref{prop:S1} (iv) and the fact that $|\mbn_0|\neq |\mbn_1|$, the two modes $\mbn_0,\mbn_1$ of $\mathbf{u}(t_0)$ satisfy none of the conditions (i)--(iii) of Proposition~\ref{prop:two}, and hence have nonzero contribution at $\mbn_0+\mbn_1$.
On the other hand, 
\[ N(\mbn_0+\mbn_1)=f(\mbn_0+\mbn_1)=f(\mbn_0)+f(\mbn_1)=N(\mbn_0)+N(\mbn_1)=N(\mbn_0)+1>1, \]
which shows $\mbn_0+\mbn_1\not\in S^{conv}$.
Therefore, by \eqref{cond:Euler} there must be another pair $\{ \mbn_2,\mbn_3\}\subset S$ of distinct points which satisfies $\mbn_0+\mbn_1=\mbn_2+\mbn_3$ and creates nonzero contribution at $\mbn_0+\mbn_1$ through nonlinear interaction.
As we mentioned above, any two modes in $V(S)$ do not interact, and hence one of $\mbn_2$ and $\mbn_3$, say $\mbn_2$, is not a vertex of $S^{conv}$.
This implies that either $f(\mbn_2)<f(\mbn_0)$ or $\mbn_2\in \ti{F}'$.
But we have $f(\mbn_2)+f(\mbn_3)=f(\mbn_0)+f(\mbn_1)$ and $f(\mbn_3)\le 1=f(\mbn_1)$, so the only possible situation is that $\mbn_2\in \ti{F}'$ and $f(\mbn_3)=1$ (i.e., $\mbn_3\in F$).
We next consider the equality $g(\mbn_2)+g(\mbn_3)=g(\mbn_0)+g(\mbn_1)=g(\mbn_0)$.
Since $\mbn_3\in F$, we have $g(\mbn_3)\ge 0$ and then $g(\mbn_2)\le g(\mbn_0)$, which combined with $\mbn_2\in \ti{F}'$ leads to $g(\mbn_2)=g(\mbn_0)$ and $g(\mbn_3)=0$, namely, $\mbn_2\in \ti{E}'$ and $\mbn_3\in E$.
Now, Proposition~\ref{prop:S1} (ii) shows that $\mbn_1,\mbn_3$ must be two endpoints of $E$, and in particular $|\mbn_1-\mbn_3|=|E|$.
This, however, implies that 
\[ |E|=|\mbn_0-\mbn_2|\le |\ti{E}'|=N(\mbn_0)\ti{N}(\mbn_0)|E|<|E|, \]
which is a contradiction.

We therefore conclude that $S\setminus V(S)=\emptyset$.
\end{proof}

By Propositions~\ref{prop:BV} and \ref{prop:interior}, we see that any two modes of $\mathbf{u}(t)$ do not interact if $t\in I\setminus I_0$.
In particular, by \eqref{cond:Euler} we have $\p _t\mathbf{u}_\mbn(t)=\mbz$ for any $t\in I\setminus I_0$ and $\mbn\in S$.
Since $I\setminus I_0$ is dense in $I$ and $\mathbf{u}_\mbn$ is smooth in $t$, we have $\p _t \mathbf{u}\equiv \mbz$ on $I\times \R ^3$, and therefore $\mathbf{u}$ is stationary.
Finally, in view of Proposition~\ref{prop:BV} and Remark~\ref{rem:SIP} (iii), it is a Beltrami flow.

This is the end of the proof of Theorem~\ref{thm:main2}, and hence, Theorem~\ref{thm:main}.


\section{Viscosity and Coriolis effect taken into account}
\label{sec:NSC}

As a generalization of Theorem~\ref{thm:main}, let us consider characterization of finite-mode solutions ($\mathbf{u},p)$ to the following Navier-Stokes-Coriolis equations:
\eq{NSC}{
\p _t\mathbf{u}-\nu \Delta \mathbf{u} +\Omega \mbe^3\gaiseki \mathbf{u} + (\mathbf{u}\cdot \nabla )\mathbf{u}+\nabla p=\mbz ,\qquad \nabla \cdot \mathbf{u}=0,
}
where $\nu ,\Omega \in \R$ and $\mbe^3=(0,0,1)$.
(The sign of $\nu$ is not relevant, since finite-mode solutions are real analytic with the radius of analyticity equal to infinity.)
To remove the zero mode $\mathbf{u}_\mbz (t)$, we first observe that it satisfies 
\eq{eq:zeromode}{\partial_t\mathbf{u}_\mbz + \Omega \mbe^3\gaiseki \mathbf{u}_\mbz = \mbz .}
Given the initial value $\mbu _{\mbz ,*}=(u_{\mbz,*}^1,u_{\mbz,*}^2,u_{\mbz,*}^3)\in \R^3$ at some $t=t_*$, this can be explicitly solved:
\[ \mathbf{u}_\mbz (t)=\Big( u_{\mbz,*}^1\cos (\Omega (t\!-\!t_*)) +u_{\mbz,*}^2\sin (\Omega (t\!-\!t_*)), \, u_{\mbz,*}^2\cos (\Omega (t\!-\!t_*))-u_{\mbz,*}^1\sin (\Omega (t\!-\!t_*)),\,u_{\mbz,*}^3\Big) .\]
Using this, we define the new velocity field $\mathbf{v}$ and pressure $q$ by
\[ \mathbf{v}(t,\mbx)=\mathbf{u}\Big( t,\mbx+\int _{t_*}^t\mathbf{u}_\mbz (\tau )d\tau \Big) -\mathbf{u}_\mbz (t),\qquad q(t,\mbx) =p\Big( t,\mbx+\int _{t_*}^t\mathbf{u}_\mbz (\tau )d\tau \Big) ,\]
which is again a finite-mode solution of \eqref{NSC} and does not have the zero mode.
Conversely, for any given mean-zero solution $(\mathbf{v},q)$ of \eqref{NSC} and any prescribed zero mode $\mathbf{u}_\mbz (t)$ satisfying \eqref{eq:zeromode}, we obtain a solution $(\mathbf{u},p)$ by the inverse transformation (defined by simply replacing $\mathbf{u}_\mbz (t)$ with $-\mathbf{u}_\mbz (t)$).
Therefore, as for the Euler equations \eqref{Euler}, it suffices to consider characterization of mean-zero solutions; namely, solutions in $\mathcal{H}_I$.
Furthermore, we apply the projection $\hat{\HP}_\mbn$ to the equation for the Fourier coefficient vector $\mathbf{u}_\mbn(t)$ to obtain
\begin{equation}\label{cond:NSC}
\begin{gathered}
\p _t\mathbf{u}_\mbn + \nu |\mbn|^2 \mathbf{u}_\mbn + \Omega J_\mbn\mathbf{u}_\mbn + \frac{i}{2}\hat{\HP}_\mbn \sum _{\mat{\mbn_1,\mbn_2\in S\\ \mbn_1+\mbn_2=\mbn}}\big[ (\mathbf{u}_{\mbn_1}\cdot \mbn_2)\mathbf{u}_{\mbn_2}+(\mathbf{u}_{\mbn_2}\cdot \mbn_1)\mathbf{u}_{\mbn_1}\big] =\mbz ,\\
\mbn\cdot \mathbf{u}_\mbn=0,\qquad t\in I,\quad \mbn\in \ti{S}=S\cup (S+S)\setminus \{ \mbz \} ,
\end{gathered}
\end{equation}
where 
\[ J_\mbn = \frac{n^3}{|\mbn|^2}\begin{pmatrix} 0 & -n^3 & n^2 \\[-5pt] n^3 & 0 & -n^1 \\[-5pt] -n^2 & n^1 & 0 \end{pmatrix} ,\qquad \mbn=(n^1,n^2,n^3)\in \mathbb{R}^3\setminus \{ \mbz \} \]
is the matrix given by $J_\mbn\mathbf{v}=\hat{\HP}_\mbn \big[ \mbe^3\gaiseki (\hat{\HP}_\mbn\mathbf{v})\big]$ ($\mathbf{v}\in \mathbb{C}^3$).
Given a solution $\{ \mathbf{u}_\mbn(t)\} _{\mbn\in S}$, we can recover the Fourier coefficient $\{ p_\mbn(t)\}_{\mbn\in S_p}$ for the pressure $p$ by
\[ p_\mbn(t)=\frac{1}{|\mbn|^2}\Big( i\Omega \big( n^2u_\mbn ^1(t)-n^1u_\mbn^2(t)\big) -\sum _{\mat{\mbn_1,\mbn_2\in S\\ \mbn_1+\mbn_2=\mbn}}(\mathbf{u}_{\mbn_1}(t)\cdot \mbn _2)(\mathbf{u}_{\mbn_2}(t)\cdot \mbn _1)\Big) ,\qquad t\in I,~~\mbn\in \ti{S},\]
where $\mathbf{u}_\mbn(t)=(u_\mbn^1(t),u_\mbn^2(t),u_\mbn^3(t))$.
In this way, as in the Euler case, characterization of finite-mode solutions of \eqref{NSC} is reduced to that of solutions of \eqref{cond:NSC} belonging to $\mathcal{H}_I$.

We shall prove the following theorem for \eqref{cond:NSC}, which includes Theorem~\ref{thm:main} as a special case $\nu =\Omega =0$.
\begin{thm}\label{thm:NSC}
Let $I\subset \R$ be an open interval
.
Then, $\mathbf{u}=\mathbf{u}(t,\mbx)\in \Sc{H}_I$ is a solution of \eqref{cond:NSC} on $I$ if and only if $\mbu$ evolves according to the linear equation $\p _t\mathbf{u}_\mbn + \nu |\mbn|^2 \mathbf{u}_\mbn + \Omega J_\mbn\mathbf{u}_\mbn =\mbz$ 
and satisfies one of the following for some $t_*\in I$ (or equivalently, for any $t_*\in I$):
\begin{enumerate}
\item The Fourier support $S$ of $\mathbf{u}(t_*)\in \mathcal{H}$ is a subset of a line passing through the origin.
\item $S$ has two linearly independent points and is a subset of a plane $P$ containing the origin.
Moreover, one of the following holds:
\begin{enumerate}
\item $\Omega \!=\!0$ or $P\!=\!P_3:=\{ \mbn \!\in \!\mathbb{R}^3:\mbn \cdot \mbe^3\!=\!0\}$, and $\mathbf{u}(t_*)$ is perpendicular to $P$ everywhere.
\item $\mbu (t_*)=\mbu^\parallel (t_*)+u^\perp (t_*)\mbe^\perp$, where $\mbe^\perp$ is (one of) the unit normal vector to $P$ and $\mbu^\parallel (t_*)$, $u^\perp (t_*)$ satisfy the following:
\begin{itemize}
\item $\mbu^\parallel (t_*)\in \mathcal{H}$ is parallel to $P$ everywhere, and its Fourier support is a subset of a circle on $P$ (with radius $\lambda >0$) centered at the origin and contains at least four points.
\item $u^\perp (t_*,\mbx )=Q(\omega (\mbx))-\big\langle Q(\omega )\big\rangle$, where $\omega (\mbx)$ is the scalar function defined by $\nabla \gaiseki \mbu^\parallel (t_*)=\lambda \omega \mbe^\perp$ and 
\[ \qquad\qquad\qquad\qquad \begin{cases}
\text{$Q$ is a polynomial with real coefficients} &\text{if $\Omega =0$ or $P=P_3$, and $\nu =0$;}\\
\text{$Q(\omega )=\kappa \omega$ for some $\kappa \in \mathbb{R}$} &\text{if $\Omega =0$ or $P=P_3$, and $\nu \neq 0$;}\\
\text{$Q(\omega )$ is either $\omega$ or $-\omega$} &\text{if $\Omega \neq 0$ and $P\neq P_3$.}
\end{cases}\] 
\end{itemize}
\end{enumerate}
\item $S$ has three linearly independent points and is a subset of a sphere centered at the origin, and $\mathbf{u}(t_*)$ is a Beltrami flow.
\end{enumerate}
\end{thm}

\begin{rem}
As is evident from the proof below, the same result holds if the viscosity $-\nu \Delta$ of the equation \eqref{NSC} is replaced by the fractional Laplacian $\nu (-\Delta)^\alpha$ with any $\alpha >0$.
\end{rem}
\begin{rem}\label{rem:thm:NSC}
Let us confirm that each of the conditions (i)--(iii) holds for any $t_*\in I$ once it holds some $t_*$, provided that $\mbu (t)$ evolves linearly: 
\eq{linearsol}{\mathbf{u}_\mbn(t)=e^{-\nu (t-t_*)|\mbn |^2}e^{-\Omega (t-t_*)J_\mbn}\mathbf{u}_\mbn(t_*),\qquad t\in I,~\mbn\in S.}
Note that the Fourier support of $\mbu (t)\in \mathcal{H}$ is equal to that of $\mbu\in \mathcal{H}_I$ for any $t\in I$.
Then, the property (i) is preserved under the flow.
Next, we see that the Coriolis term $\Omega J_\mbn \mbu_\mbn$ vanishes when $\Omega =0$ or $\mbn \in P_3$.
Therefore, the property (ii)-(a) is also preserved.
If $\Omega =0$ or $P=P_3$ and $\nu =0$, $\mbu (t)$ is independent of $t$.
In the other cases, the condition (ii)-(b) or (iii) implies that all points $\mbn$ of the Fourier support of $\mbu$ are equidistant from the origin, and so the effect of the viscosity term is simply multiplication by the same constant for all Fourier modes.
Then, (ii)-(b) is preserved if $\Omega =0$ or $P=P_3$ (and $\nu \neq 0$).
Finally, we notice that the additional condition $Q(\om )=\pm \om$ in (ii)-(b) is equivalent to saying that $\mbu^\parallel (t_*)+Q(\om )\mbe^\perp$ is a Beltrami flow.
Since the effect of the Coriolis term on each Fourier coefficient vector $\mbu_\mbn$ is simply rotation on the plane perpendicular to $\mbn$ (see \eqref{eq:Coriolis} and following comments
), in view of the characterization Lemma~\ref{lem:Beltrami}, a Beltrami flow is changed to another Beltrami flow (of the same sign) by the linear evolution \eqref{linearsol}.

The fact we have just verified is sufficient to conclude the \emph{if} part of Theorem~\ref{thm:NSC}, because each of the conditions (i)--(iii) at some time $t_*\in I$ implies that there is no nonlinear interaction (i.e., the last term on the left-hand side of \eqref{cond:NSC} vanishes) at $t=t_*$ (see the proof of the \emph{if} part of Theorem~\ref{thm:main}).
\end{rem}

In the rest of this section, we shall prove the \emph{only if} part of Theorem~\ref{thm:NSC}. 
We first recall the following property of the matrix $J_\mbn$ (see, e.g., \cite[Section~2]{BMN99}):
\eq{eq:Coriolis}{\mbn\cdot \mathbf{u}_\mbn=0 \quad \Rightarrow \quad \left\{ \begin{aligned}
J_\mbn\mathbf{u}_\mbn&=\frac{n^3}{|\mbn|}\Big( \frac{\mbn}{|\mbn|}\gaiseki \mathbf{u}_\mbn\Big) ,\\
e^{-\Omega tJ_\mbn}\mathbf{u}_\mbn &=\cos \Big( \Omega t\frac{n^3}{|\mbn|}\Big) \mathbf{u}_\mbn-\sin \Big( \Omega t\frac{n^3}{|\mbn|}\Big) \Big( \frac{\mbn}{|\mbn|}\gaiseki \mathbf{u}_\mbn\Big) .\end{aligned}\right.
}
In particular, we see the following:
\begin{itemize}
\item If $\mbn\in P_3=\{ \mbn :n^3=0\}$ and $\mbn\cdot \mathbf{u}_\mbn=0$, then $J_\mbn\mathbf{u}_\mbn=0$ and $e^{-\Omega tJ_\mbn}\mathbf{u}_\mbn=\mathbf{u}_\mbn$.
\item If $\mbn \not\in P_3$, then on $\{ \mathbf{v}\in \mathbb{C}^3:\mbn\cdot \mathbf{v}=0\}$, $J_\mbn$ acts as rotation by $\pm \pi /2$ plus multiplication by a non-zero real constant $n^3/|\mbn|$, and $e^{-\Omega tJ_\mbn}$ acs as rotation by $\pm \Omega tn^3/|\mbn|$.
\item $J_\mbn$ and $e^{-\Omega tJ_\mbn}$ commute with the operation $i\mbn \times$.
In particular, each of the sets of positive and negative Beltrami vectors $BV^\pm$ at $\mbn$ is invariant under $J_\mbn$ and $e^{-\Omega tJ_\mbn}$.
\end{itemize}

\begin{proof}[Proof of the \emph{only if} part of Theorem~\ref{thm:NSC}]
When $S$ is a subset of a line, there is no nonlinear interaction, i.e., the nonlinear term in \eqref{cond:NSC} vanishes for each $\mbn$.
Then, the evolution for each coefficient vector $\mathbf{u}_\mbn(t)$ is decoupled and becomes linear, verifying the case (i).
When $S$ has three directions, we see that $\mathbf{u}(t)$ must be a Beltrami flow for each $t$, by exactly the same reduction as given in Section~\ref{sec:S}.
This again means that the nonlinear term in \eqref{cond:NSC} vanishes and each mode evolves linearly, which gives the case (iii).
Also, when $S$ is planar, $\nu =0$, $\Omega \neq 0$ but $P=P_3$, the Coriolis term $\Omega \mbe^3\gaiseki \mathbf{u}$ vanishes identically, which leads to the same situation as for the Euler equation \eqref{cond:Euler}.
Hence, the only situation requiring additional consideration is the planar case (ii), with either $\nu \neq 0$ or ``$\Omega \neq 0$, $P\neq P_3$''.
As we will see below, a major part of the argument in Section~\ref{sec:2D} remains valid in the present context.

First, assume that $\mathbf{u}$ is perpendicular to $P$ everywhere at every time.
Since there is no nonlinear interaction for such a solution, each Fourier mode evolves linearly.
This excludes the case that $\Omega \neq 0$ and $P\neq P_3$, for which the coefficient vector at any mode $\mbn$ with $n^3\neq 0$ is rotating and cannot be perpendicular to $P$ for all time.

Hereafter, we assume that the horizontal component of $\mathbf{u}$ (the component parallel to $P$) is not identically zero.
As in Lemma~\ref{lem:two}, we write 
\[ \mathbf{u}_\mbn(t)=u_\mbn^\parallel (t)\mbe^\parallel_\mbn +u_\mbn^\perp (t)\mbe^\perp \quad (\mbn\in S),\qquad \mbe^\parallel_\mbn :=\mbe^\perp \gaiseki \tfrac{n}{|n|},\]
so that the set $S_\parallel :=\{ \mbn\in S: u_\mbn^\parallel \not\equiv 0\}$ is non-empty and symmetric.
Taking the horizontal component of the equation \eqref{cond:NSC} and setting $\mbu^\parallel_\mbn (t):=u^\parallel_\mbn (t)\mbe^\parallel_\mbn$, we obtain that
\eq{cond:parallel2}{\p_t\mbu^\parallel _\mbn +\nu |\mbn|^2\mbu^\parallel_\mbn -\Omega \frac{n^3}{|\mbn|}u^\perp _\mbn \mbe^\parallel_\mbn +\frac{i}{2}\hat{\HP}_{\mbn}\sum _{\mat{\mbn_1,\mbn_2\in S_\parallel \\ \mbn_1+\mbn_2=\mbn}}\big[ (\mbu^\parallel_{\mbn_1}\cdot \mbn_2)\mbu^\parallel _{\mbn_2}+(\mbu^\parallel_{\mbn_2}\cdot \mbn_1)\mbu^\parallel _{\mbn_1}\big] =\mbz ,
}
which will be used as the counterpart of \eqref{cond:parallel}.

We begin with observing that $S_\parallel$ is not contained in a line (and hence it has at least four points).
In fact, this can be shown in exactly the same way as for the Euler case; recall the contradiction argument for proving Proposition~\ref{prop:planar}(i).
Now, the theorem follows once the following claims (1)--(3) are verified:
\begin{itemize}
\item[(1)] $S_\parallel$ is contained in a circle centered at the origin.
\item[(2)] The set $S_\perp :=\{ \mbn\in S: u_\mbn^\perp \not\equiv 0\}$ is a subset of $S_\parallel$ (and hence $S=S_\parallel$).
\item[(3)] Each $\mathbf{u}_\mbn(t)$ evolves linearly, and  
\[ \qquad \text{$u_\mbn^\perp(t_*)=i\kappa u_\mbn^\parallel(t_*)$ for some} \left\{ \begin{alignedat}{2} &\kappa \in \R &\quad &\text{in \textbf{Case A}: $\Omega =0$ or $P=P_3$, and $\nu \neq 0$},\\
&\kappa \in \{ \pm 1\} & &\text{in \textbf{Case B}: $\Omega \neq 0$ and $P\neq P_3$}. \end{alignedat} \right. \]
\end{itemize}
%
We shall prove these claims for the two \textbf{Cases A} and \textbf{B} separately.

\medskip
{\bf Case A}: $\Omega =0$ or $P=P_3$, and $\nu \neq 0$.

In this case, the third term on the left-hand side of the equation \eqref{cond:parallel2} vanishes, so the equation for the horizontal component $\mbu^\parallel (t)$ is decoupled from that for the vertical part $u^\perp (t)$.
Then, the claim (1) can be verified by the same argument as in the Euler case; see the proof of Proposition~\ref{prop:planar}(i).
In particular, there is no pair of frequencies in $S$ creating the horizontal component through the nonlinear interaction, so the horizontal component evolves linearly.
This allows us to set the horizontal component of $\mathbf{u}(t,\mbx)$ as follows:
\[ \mathbf{u}^\parallel (t,\mbx)=e^{-\nu \lambda ^2(t-t_*)}\mathbf{u}^\parallel (t_*,\mbx):=e^{-\nu \lambda ^2(t-t_*)}\sum _{j=0}^{p-1}\alpha _j\mbe^\parallel_{\mbn_j}e^{i\mbn_j\cdot \mbx}, \]
where $\lambda >0$ is the radius of the circle containing $S_\parallel$, $p\geq 4$ is an even integer, $\mbn_0,\dots ,\mbn_{p-1}$ denote all the points of $S_\parallel$ located in this order, and $\alpha _j\in \mathbb{C}\setminus \{ 0\}$ satisfies $i\alpha _{j+(p/2)}=\overline{i\alpha_j}$ for $0\leq j<p/2$.
The equation for the vertical component $u^\perp (t,\mbx)=\sum _{\mbn\in S_\perp}u^\perp _\mbn(t)e^{i\mbn\cdot \mbx}$ is given by
\eq{cond:NSCp}{
\p _tu^\perp -\nu \Delta u^\perp +e^{-\nu \lambda ^2(t-t_*)}\big( \mathbf{u}^\parallel (t_*)\cdot \nabla \big) u^\perp =0,&\qquad t\in I,}
which is the counterpart of \eqref{cond:Euler^h} in the Euler case.

In order to prove the claims (2) and (3), we assume $u^\perp \not\equiv 0$ and set $q:=\max \{ N(\mbn ):\mbn \in S_\perp \} >0$, where $N$ is the Minkowski functional of $S^{conv}_\parallel$ on the plane $P$.
Then, exactly the same argument as for Lemma~\ref{lem:planar1} shows that $q$ must be a positive integer and $S_\perp \cap \p [qS^{conv}_\parallel ]=\{ \ti{\mbn}_{j,k} : j=0,\dots ,p-1,\,k=0,\dots ,q-1 \}$, where $\ti{\mbn}_{j,k}=(q-k)\mbn_j+k\mbn_{j+1}$. 
In contrast to the Euler case, it will turn out that $q>1$ is prohibited due to the presence of the viscosity term in \eqref{cond:NSCp}.

Following the proof of Lemma~\ref{lem:planar2}, we first obtain the same expression of $\{ u^\perp _{\ti{\mbn}_{j,k}}\} _{j,k}$ as in the Euler case:
\[ u^\perp _{\ti{\mbn}_{j,k}}(t)=\begin{pmatrix} q \\ k\end{pmatrix} (i\alpha_j)^{q-k}(i\alpha _{j+1})^{k}\beta (t),\qquad 0\leq j\leq p-1,\quad 0\leq k\leq q-1,\quad t\in I,\]
for some non-zero smooth function $\beta :I\to \R$.
Similarly as before, the function $\om (\mbx)$ satisfying $\nabla \gaiseki \mathbf{u}^\parallel (t_*,\mbx)=\lambda \om (\mbx)\mbe^\perp$ is explicitly given by $\om (\mbx)=\sum _{j=0}^{p-1}i\alpha_je^{i\mbn_j\cdot \mbx}$.
Recalling that
\[ \om (\mbx)^q=\sum _{j=0}^{p-1}\sum _{k=0}^{q-1}\begin{pmatrix} q \\ k \end{pmatrix}(i\al _j)^{q-k}(i\al _{j+1})^ke^{i\ti{\mbn}_{j,k}\cdot \mbx}+\eta (\mbx)\]
for some function $\eta $ whose Fourier support is finite and contained in $\{ \mbn\in P:N(\mbn)<q\}$, we introduce a new function $\zeta (t,\mbx):=u^\perp (t,\mbx)-\beta (t)\om (\mbx)^q$ whose Fourier support is contained in $\{ \mbn\in P:N(\mbn)<q\}$.
From \eqref{cond:NSCp} and that $(\mathbf{u}^\parallel (t_*)\cdot \nabla )\om ^q=0$, 
\[ \p_t \zeta -\nu \Delta \zeta +e^{-\nu \lambda^2(t-t_*)}(\mathbf{u}^\parallel (t_*)\cdot \nabla )\zeta = -\beta'(t)\om^q+\nu \beta (t)\Delta [\om ^q].\]
Since all the terms but the last one on the left-hand side have Fourier supports in $qS^{conv}_\parallel$, in the same manner as we did in the proof of Lemma~\ref{lem:planar2}, we can show that the Fourier support of $\zeta$ is actually contained in $(q-1)S^{conv}_\parallel$ (and $\zeta \equiv 0$ if $q=1$).
Comparing the Fourier coefficients of the both sides at $\ti{\mbn}_{0,0}$ as before, we have
\[ 0=-\beta'(t)(i\alpha_0)^q+ \nu \beta (t) (-q^2\lambda^2)(i\alpha_0)^q.\]
Solving this, we have $\beta (t)=\beta _*e^{-\nu q^2\lambda^2(t-t_*)}$ for some $\beta _*\in \mathbb{R}\setminus \{ 0\}$, which gives 
\eq{eq:zeta-2}{\p_t \zeta -\nu \Delta \zeta +e^{-\nu \lambda^2(t-t_*)}(\mathbf{u}^\parallel (t_*)\cdot \nabla )\zeta = \nu \beta _*e^{-\nu q^2\lambda^2(t-t_*)}\big( q^2\lambda^2\om^q+\Delta [\om ^q]\big) .}

Now, we shall prove $q=1$.
Suppose $q>1$, then the Fourier coefficient of the right-hand side of \eqref{eq:zeta-2} does not vanish at $\ti{\mbn}_{j,k}$ for any $0\leq j\leq p-1$ and $1\leq k\leq q-1$; in fact, it is given by
\[ \nu \beta _*e^{-\nu q^2\lambda^2(t-t_*)}\begin{pmatrix} q \\ k \end{pmatrix}(i\al _j)^{q-k}(i\al _{j+1})^k \big( |\ti{\mbn}_{j,0}|^2-|\ti{\mbn}_{j,k}|^2\big) . \]
On the other hand, the Fourier coefficient of the left-hand side of \eqref{eq:zeta-2} at these frequencies can be computed as
\eqq{
&e^{-\nu \lambda^2(t-t_*)}\Big( \mathbf{u}^\parallel_{\mbn_j}(t_*)\cdot i\big[ (q-k-1)\mbn_j+k\mbn_{j+1}\big]\Big) \zeta _{(q-k-1)\mbn_j+k\mbn_{j+1}}\\
&\qquad\qquad +e^{-\nu \lambda^2(t-t_*)}\Big( \mathbf{u}^\parallel_{\mbn_{j+1}}(t_*)\cdot i\big[ (q-k)\mbn_j+(k-1)\mbn_{j+1}\big]\Big) \zeta _{(q-k)\mbn_j+(k-1)\mbn_{j+1}}\\
&= e^{-\nu \lambda^2(t-t_*)}\big[ (\mbn_j \gaiseki \mbn_{j+1})\cdot \mbe^\perp\big] \Big( i\alpha_jk\zeta _{(q-k-1)\mbn_j+k\mbn_{j+1}}-i\alpha_{j+1}(q-k)\zeta _{(q-k)\mbn_j+(k-1)\mbn_{j+1}}\Big) .
}
(Note that the only possible pairs $(\mbn',\ti{\mbn}')$ with $\mbn'\in S_\parallel$, $\ti{\mbn}'\in (q-1)S^{conv}_\parallel$, $\mbn'+\ti{\mbn}'=\ti{\mbn}_{j,k}$ are $(\mbn_j,(q-k-1)\mbn_j+k\mbn_{j+1})$ and $(\mbn_{j+1},(q-k)\mbn_j+(k-1)\mbn_{j+1})$.)
Comparing the Fourier coefficients of both sides of \eqref{eq:zeta-2} at these frequencies, we obtain the equality
\eqq{
&i\alpha_jk\zeta _{(q-k-1)\mbn_j+k\mbn_{j+1}}-i\alpha_{j+1}(q-k)\zeta _{(q-k)\mbn_j+(k-1)\mbn_{j+1}}\\
&=\nu \beta _*e^{-\nu (q^2-1)\lambda^2(t-t_*)}\begin{pmatrix} q \\ k \end{pmatrix}(i\al _j)^{q-k}(i\al _{j+1})^k\frac{|\ti{\mbn}_{j,0}|^2-|\ti{\mbn}_{j,k}|^2}{(\mbn_j\gaiseki \mbn_{j+1})\cdot \mbe^\perp}
}
for each $0\leq j\leq p-1$ and $1\leq k\leq q-1$.
Noticing $\big( \begin{smallmatrix} q\\ k\end{smallmatrix}\big)=\frac{q}{q-k}\big( \begin{smallmatrix} q-1\\ k\end{smallmatrix}\big) =\frac{q}{k}\big( \begin{smallmatrix} q-1\\ k-1\end{smallmatrix}\big)$, we divide both sides of the above equality by $\frac{k(q-k)}{q}\big( \begin{smallmatrix} q \\ k \end{smallmatrix}\big) (i\al _j)^{q-k}(i\al _{j+1})^k$ to have
\eqq{
&\frac{\zeta _{(q-k-1)\mbn_j+k\mbn_{j+1}}}{\big( \begin{smallmatrix} q-1 \\ k \end{smallmatrix}\big) (i\al _j)^{q-k-1}(i\al _{j+1})^k}-\frac{\zeta _{(q-k)\mbn_j+(k-1)\mbn_{j+1}}}{\big( \begin{smallmatrix} q-1 \\ k-1 \end{smallmatrix}\big) (i\al _j)^{q-k}(i\al _{j+1})^{k-1}}\\
&\quad =\nu \beta _*e^{-\nu (q^2-1)\lambda^2(t-t_*)}\frac{q}{k(q-k)}\frac{|\ti{\mbn}_{j,0}|^2-|\ti{\mbn}_{j,k}|^2}{(\mbn_j\gaiseki \mbn_{j+1})\cdot \mbe^\perp}.
}
Summing up these equalities for $0\leq j\leq p-1$ and $1\leq k\leq q-1$, we obtain
\[ 0=q\nu \beta _*e^{-\nu (q^2-1)\lambda^2(t-t_*)}\sum _{j=0}^{p-1}\sum _{k=1}^{q-1}\frac{|\ti{\mbn}_{j,0}|^2-|\ti{\mbn}_{j,k}|^2}{k(q-k)\big[ (\mbn_j\gaiseki \mbn_{j+1})\cdot \mbe^\perp \big]}.\]
Since all the terms in the sum are non-zero and have the same sign, this is a contradiction. 
Therefore, it holds that $q=1$ and $\zeta \equiv 0$, and thus $u^\perp (t_*,\mbx )=\beta _*\om (\mbx )$ for some $\beta _*\in \mathbb{R}\setminus \{ 0\}$.
This shows the claims (2) and (3).

\medskip
{\bf Case B}: $\Omega \neq 0$ and $P\neq P_3$.

We have to modify the previous argument for \textbf{Case A}, since the equation \eqref{cond:parallel2} for $\mbu^\parallel$ now contains $u^\perp$ as well.
We again consider the set $S^{conv}_\parallel$ and the corresponding Minkowski functional $N$ on $P$.
We shall derive $q:=\max \{ N(\mbn):\mbn\in S_\perp \} \leq 1$ before proving the claim (1).

Suppose for contradiction that $q>1$.
Then, since $\p [qS^{conv}_\parallel]$ contains at least one point of $S_\perp$, there is a side $E$ of the polygon $S^{conv}_\parallel$ such that $S_\perp \cap [qE] \neq \emptyset$.
Let $\mbn_1,\mbn_2,\dots ,\mbn_p$ be all the points of $S_\parallel$ located on $E$ in this order; hence $\mbn_1,\mbn_p$ are the endpoints of $E$.
Since $P\neq P_3$, we may assume that $\mbn_1\not\in P_3$.
Let $\mbn$ be the point of $S_\perp \cap [qE]$ which is the closest to the endpoint $q\mbn_1$, and we claim that $\mbn=q\mbn_1$.
Suppose this is not the case, then the frequency $\mbn$ interacts with $\mbn_1\in S_\parallel$ to create non-zero vertical component at $\mbn+\mbn_1$.
Since $\mbn+\mbn_1\not\in S$, there must exist a pair $(\mbn',\ti{\mbn}')\in S_\parallel \times S_\perp$ such that 
$(\mbn',\ti{\mbn}')\neq (\mbn_1,\mbn)$ and $\mbn'+\ti{\mbn}'=\mbn_1+\mbn$.
By an argument using the linear functional which is equal to $1$ on $E$, we see $\mbn'\in E$ and $\ti{\mbn}'\in qE$.
It then holds that $\mbn'\in \{ \mbn_2,\dots ,\mbn_p\}$, which however implies that $\ti{\mbn}'$ is either out of the segment $qE$ or closer to $q\mbn_1$ than $\mbn$, contradicting the definition of $\mbn$.
We therefore verify that $q\mbn_1\in S_\perp$.

We next claim $q<2$.
To see this, observe that $\Omega J_{q\mbn_1}\mathbf{u}_{q\mbn_1}$ has non-zero horizontal component (note that $q\mbn_1\not\in P_3$).
Since $q\mbn_1\not\in S_\parallel$, from the equation \eqref{cond:NSC} it must be canceled out by the interaction of two frequencies $\mbn',\mbn''$ in $S_\parallel$.
This implies that $q=N(q\mbn_1)=N(\mbn'+\mbn'')\leq N(\mbn')+N(\mbn'')\leq 2$.
Furthermore, if $q=2$, then the only possible $\mbn',\mbn''\in S_\parallel$ with $\mbn'+\mbn''=2\mbn_1$ is $\mbn'=\mbn''=\mbn_1$ (this can be shown by using a linear functional $f$ with $f(\mbn_1)=1$ and $f<1$ on $S^{conv}_\parallel\setminus \{ \mbn_1\}$), which however do not interact.
Therefore, it must hold that $q<2$.

We can then deduce that $S^{conv}_\parallel$ is inscribed in a circle on $P$ centered at the origin and there is no point of $S_\parallel$ on $\p S^{conv}_\parallel$ other than the vertices.
To show this, it suffices to verify that any vertex $\mbn_0$ of $S^{conv}_\parallel$ and an adjacent point $\mbn_0'\in S_\parallel \cap \p S^{conv}_\parallel$ satisfy $|\mbn_0|=|\mbn_0'|$.
In fact, if $|\mbn_0|\neq |\mbn_0'|$, then non-zero horizontal component would be created at $\mbn_0+\mbn_0'$.
Since $N(\mbn_0+\mbn_0')=2$ and $q<2$ imply $\mbn_0+\mbn_0'\not\in S$, and that there is no other pair of frequencies in $S_\parallel$ creating the mode $\mbn_0+\mbn_0'$, this component could not be canceled out, which contradicts the equation \eqref{cond:NSC}. 

So far, we have shown $1<q<2$, $q\mbn_1\in S_\perp$, and $S_\parallel\cap E=\{ \mbn_1,\mbn_2\}$.
We can now derive a contradiction by an argument similar to the proof of Lemma~\ref{lem:planar1}.
Indeed, the vertical component at $q\mbn_1+\mbn_2\not\in S$ created from $q\mbn_1\in S_\perp$ and $\mbn_2\in S_\parallel$ could be canceled only by the contribution from $(q-1)\mbn_1+\mbn_2$ and $\mbn_1$.
The contribution from $(q-1)\mbn_1+\mbn_2\in S_\perp$ and $\mbn_2\in S_\parallel$ must in turn be canceled with that from $(q-2)\mbn_1+2\mbn_2$ and $\mbn_1$.
However, since $(q-2)\mbn_1+2\mbn_2\not\in qS^{conv}_\parallel$, this is a contradiction.
We therefore confirm that $q\leq 1$.

Now that there is no point of $S$ in $\{ \mbn \in P:N(\mbn)>1\}$, we can prove the claim (1) by following the proof of Proposition~\ref{prop:planar}(i).
We first see that $S_\parallel$ is inscribed in a circle and that $S_\parallel \cap \p S^{conv}_\parallel$ consists of the vertices of $S^{conv}_\parallel$, which is verified by the same argument as for the claim (a) in the proof of Proposition~\ref{prop:planar}(i).
Next, we recall the argument for the claim (b) in the proof of Proposition~\ref{prop:planar}(i) to show that there is no point of $S_\parallel$ in the interior of $S^{conv}_\parallel$, and hence the claim (1).

Moreover, we follow the proof of Lemma~\ref{lem:planar1} (and the first half of the proof of Lemma~\ref{lem:planar2}) to see that $q=1$, $S_\perp \cap \p S^{conv}_\parallel =S_\parallel$ (unless $u^\perp \equiv 0$) and
\eq{relation-1}{u^\parallel_{\mbn_{j+1}}(t)u^\perp _{\mbn_j}(t)=u^\parallel _{\mbn_j}(t)u^\perp _{\mbn_{j+1}}(t),\qquad 0\leq j\leq p-1,\quad t\in I.}
Since each of $u^\parallel_{\mbn_j}(t)$, $u^\perp_{\mbn_j}(t)$ is real analytic and not identically zero on $I$ (unless $u^\perp \equiv 0$), the above relations \eqref{relation-1} are easily extended to 
\[ u^\parallel_{\mbn_j}(t)u^\perp _{\mbn_{j'}}(t)=u^\parallel _{\mbn_{j'}}(t)u^\perp _{\mbn_j}(t),\qquad 0\leq j,j'\leq p-1,\quad t\in I.\]
Hence, we deduce from Lemma~\ref{lem:two} that there is no interaction between any pair $\{ \mbn_j,\mbn_{j'}\}$ of two vertices of $S^{conv}_\parallel$ (not necessarily adjacent).
Using this fact, we can show $S_\perp \cap \{ \mbn \in P:N(\mbn)<1\} =\emptyset$ by a slight modification of the argument for the claim (b) in the proof Proposition~\ref{prop:planar}(i).
This verifies the claim (2).

Finally, we shall establish the claim (3).
We have already seen that there is no interaction between any pair of frequencies in $S$.
Then, all the coefficient vectors evolve linearly:
$\mathbf{u}_{\mbn}(t)=e^{-\nu \lambda ^2(t-t_*)}e^{-\Omega (t-t_*)J_{\mbn}}\mathbf{u}_{\mbn}(t_*)$; or in each component,
\eq{repr-1}{
u_\mbn^\parallel (t) &=e^{-\nu \la^2(t-t_*)}\Big\{ \cos \Big( \Omega (t-t_*)\frac{n^3}{|\mbn|}\Big) u^\parallel_\mbn(t_*) +\sin \Big( \Omega (t-t_*)\frac{n^3}{|\mbn|}\Big) u^\perp_\mbn(t_*) \Big\} ,\\
u_\mbn^\perp (t) &=e^{-\nu \la^2(t-t_*)}\Big\{ - \sin \Big( \Omega (t-t_*)\frac{n^3}{|\mbn|}\Big) u^\parallel_\mbn(t_*) +\cos \Big( \Omega (t-t_*)\frac{n^3}{|\mbn|}\Big) u^\perp_\mbn(t_*) \Big\} 
}
for $t\in I$ and $\mbn\in S$.

Here, we claim that $u^\parallel_\mbn (t)\neq 0$ for all $\mbn \in S$ and $t\in I$.
Suppose this is not the case, say $u^\parallel_{\mbn_j}(\tau )=0$ for some $j$ and $\tau \in I$.
Note that $u^\perp _{\mbn_j}(\tau )\neq 0$, since otherwise we have $\mbu _{\mbn_j}\equiv 0$ on $I$ from \eqref{repr-1}, contradicting that $\mbn_j\in S_\parallel$.
The relation \eqref{relation-1} then implies that $u^\parallel_{\mbn_{j+1}}(\tau )=0$, and similarly, $u^\parallel_{\mbn}(\tau )=0$ and $u^\perp _{\mbn}(\tau )\neq 0$ for all $\mbn\in S$.
From \eqref{repr-1}, we have
\[ u^\parallel _\mbn (t)=e^{-\nu \la^2(t-\tau )}\sin \Big( \Omega (t-\tau )\frac{n^3}{\lambda}\Big) u^\perp_\mbn(\tau ),\qquad u_\mbn^\perp (t) =e^{-\nu \la^2(t-\tau )}\cos \Big( \Omega (t-\tau )\frac{n^3}{\lambda}\Big) u^\perp_\mbn(\tau )\]
for any $\mbn\in S$, $t\in I$.
Substituting them into \eqref{relation-1}, we obtain that
\[ \sin \Big( \Omega (t-\tau )\frac{n_{j+1}^3-n_j^3}{\lambda}\Big) =0,\qquad 0\leq j\leq p-1,\quad t\in I.\]
This is, however, impossible because the assumption $P\neq P_3$ implies $n_{j+1}^3\neq n_j^3$ for some $j$.
 
Now, from \eqref{relation-1} we see that $\gamma (t):=u_\mbn^\perp (t)/u_\mbn^\parallel (t)\in \mathbb{C}$ is independent of $\mbn\in S$ for each $t\in I$.
Moreover, since $\mathbf{u}(t)$ is real-valued, it holds that $\mathbf{u}_{\mbn_{p/2}}(t)=\overline{\mathbf{u}_{\mbn_0}(t)}$; i.e., $(-u^\parallel _{\mbn_{p/2}}(t),u^\perp _{\mbn_{p/2}}(t))=(\overline{u^\parallel _{\mbn_0}(t)},\overline{u^\perp _{\mbn_0}(t)})$.
Therefore, we have
\[ \overline{\gamma (t)}=\overline{\Big( \frac{u^\perp _{\mbn_0}(t)}{u^\parallel _{\mbn_0}(t)}\Big)} =\frac{u^\perp _{\mbn_{p/2}}(t)}{-u^\parallel _{\mbn_{p/2}}(t)}=-\gamma (t),\]
which shows that $\gamma (t)=i\kappa(t)$ with $\kappa (t)\in \R$. 
Substituting the relation $u_\mbn^\perp (t)=i\kappa (t)u_\mbn^\parallel (t)$ into \eqref{repr-1}, we have
\[ -\sin \phi +i\kappa (t_*) \cos \phi = i\kappa (t) \Big\{ \cos \phi +i\kappa (t_*)\sin \phi \Big\} ,\qquad \phi :=\Omega (t-t_*)\frac{n^3}{|\mbn|}\]
for any $t\in I$ and $\mbn \in S$.
If we choose $\mbn\in S\setminus P_3$ and $t\in I$ such that $\phi \not\in \frac{\pi}{2}\mathbb{Z}$, then this implies $1=\kappa (t)\kappa (t_*)$ and $\kappa (t)=\kappa (t_*)$, and therefore $\kappa(t)\in \{ \pm 1\}$.
By the continuity of $\gamma(t)$, it holds either $\kappa (t)\equiv 1$ or $\kappa (t)\equiv -1$ on $I$.
We have thus proved the claim (3) for \textbf{Case B}.

This is the end of the proof of Theorem~\ref{thm:NSC}.
\end{proof}

\vspace{0.5cm}
\noindent
{\bf Acknowledgments.}\ 
Research of NK was partly supported by the JSPS Grant-in-Aid for Young Scientists 16K17626.
Research of TY was partly supported by the JSPS Grants-in-Aid for Scientific
Research  17H02860, 18H01136, 18H01135 and 20H01819.


\end{document}